\documentclass[a4paper,12pt,reqno]{amsart}
\usepackage{amsmath, amsthm, amsfonts, amssymb, amscd, bm}
\usepackage{fancyvrb} 
\usepackage{url}
\usepackage[utf8]{inputenc}
\usepackage{longtable} 
\usepackage{mathtools} 
\usepackage{enumitem} 

\newcommand{\reals}{{\mathbb R}}
\newcommand{\complex}{{\mathbb C}}
\newcommand{\cstar}{{\mathbb C}^*}
\newcommand{\cn}{{\mathbb C}^n}
\newcommand{\rsphere}{{\mathbb P}^1}
\newcommand{\projspace}[1]{{\mathbb P}^{#1}}
\newcommand{\pone}{{\mathbb P}^1}
\newcommand{\pseven}{{\mathbb P}^7}
\newcommand{\integers}{{\mathbb Z}}
\newcommand{\scriptO}{{\mathcal O}}
\DeclareMathOperator{\PSL}{PSL}
\DeclareMathOperator{\PSU}{PSU}

\newcommand{\Sym}[1]{{\mathfrak S}_{#1}}
\newcommand{\mbsgp}{{\PSL_2(\complex)}}

\DeclareMathOperator{\Liesl}{sl}

\theoremstyle{plain}
\newtheorem{theorem}{Theorem}[subsection]
\newtheorem{maintheorem}{Theorem}[section]
\newtheorem{lemma}[theorem]{Lemma}
\newtheorem{proposition}[theorem]{Proposition}
\newtheorem{corollary}[theorem]{Corollary}
\newtheorem{maincorollary}[maintheorem]{Corollary}

\theoremstyle{definition}
\newtheorem{definition}[theorem]{Definition}
\newtheorem{maindefinition}[maintheorem]{Definition}
\newtheorem{example}[theorem]{Example}

\theoremstyle{remark}
\newtheorem{remark}[theorem]{Remark}
\newtheorem{mainremark}[maintheorem]{Remark}

\numberwithin{equation}{subsection}

\usepackage[pdftex,
    pdfauthor={Alexander Hanysz},
    pdftitle={Holomorphic flexibility properties of the space of cubic rational maps},
    naturalnames]{hyperref} 

\begin{document}

\date{12th November 2012}
\author{Alexander Hanysz}
\address{School of Mathematical Sciences, University of Adelaide,
    Adelaide SA 5005, Australia}
\email{alexander.hanysz@adelaide.edu.au}
\title[Cubic rational maps]
      {Holomorphic flexibility properties \\
       of the space of cubic rational maps}
\begin{abstract}
  For each natural number~$d$,
  the space $R_d$ of rational maps of degree~$d$ on the Riemann sphere
  has the structure of a complex manifold.
  The topology of these manifolds has been extensively studied.
  The recent development of Oka theory
  raises some new and interesting questions about their complex structure.
  We apply geometric invariant theory to the cases of degree~2 and~3,
  studying a double action of the Möbius group on $R_d$.
  The action on $R_2$ is transitive,
  implying that $R_2$ is an Oka manifold.
  The action on $R_3$ has $\complex$ as a categorical quotient;
  we give an explicit formula for the quotient map
  and describe its structure in some detail.
  We also show that $R_3$ enjoys the holomorphic flexibility properties
  of strong dominability and $\complex$-connectedness.
\end{abstract}


\maketitle

{ \scriptsize
\tableofcontents
}

\section{Introduction and statement of results} \label{section:intro}

The space of rational maps on the Riemann sphere
can be given the structure of a complex manifold.
The topology of this manifold
(the compact-open topology)
has been studied extensively,
beginning with the work of Segal~\cite{Segal-1979}.
In this paper we study rational maps from a geometric point of view,
motivated by the recent development of Oka theory.
In particular, we are interested in the
holomorphic flexibility properties of
dominability and $\complex$-connectedness (defined below),
which can be viewed as opposite to Kobayashi hyperbolicity.

We write $R_d$ for the set of rational maps
of degree~$d$.
Each such map can be written as a quotient of two relatively prime polynomials
whose maximum degree is $d$.
The space $\scriptO(\pone,\pone)$
of all rational maps is the union of $R_d$ for $d=0,1,2,\ldots$;
the $R_d$ are exactly the connected components of this space.
Section~\ref{section:context} describes
the complex structure on $R_d$
and gives a brief overview of the relevant concepts from Oka theory.

\medskip

\textbf{Basic question:}
\textit{Is $\scriptO(\pone,\pone)$ an Oka manifold?}

\medskip

The question can be approached one degree at a time:
is each component $R_d$ an Oka manifold?
We apply geometric invariant theory,
using the results of Snow~\cite{Snow-1982}.
In particular, the Möbius group $\mbsgp$
acts on $R_d$ in two ways, by precomposition and postcomposition
(see for example Ono and Yamaguchi~\cite{Ono-Yamaguchi-2003}).
We combine these two actions into a two-sided action
of $\mbsgp\times\mbsgp$, described in Section~\ref{section:actions}.

For low degree,
we have $R_0=\pone$ and $R_1=\mbsgp$,
both of which are known examples of Oka manifolds.
For $d=2$, the two-sided group action is transitive:

\begin{maintheorem}[{\cite[Proposition~2.1]{Guest-et-al-1995}}] \label{thm:r2oka}
The space $R_2$ of rational maps of degree~2
is a complex homogeneous manifold.
\end{maintheorem}

At the end of Section~\ref{section:actions} we give
an alternative proof of this result, as an introduction
to the methods used later in this paper.
Complex homogeneous manifolds are always Oka
(see for example \cite[Proposition~5.5.1]{Forstneric-2011}),
so we have the following consequence:

\begin{maincorollary} \label{cor:r2oka}
$R_2$ is an Oka manifold.
\end{maincorollary}

For $d\geq3$ it is presently unknown whether $R_d$ is an Oka manifold.
The ideal situation would be to express $R_d$ as a holomorphic
fibre bundle whose base and fibre are Oka.
This would be sufficient to show that $R_d$ is Oka
(see for example \cite[Theorem~5.5.4]{Forstneric-2011}).
The quotient map of a group action would be a natural candidate for
such a bundle.

The nearest we can get at present
is to exhibit a group action whose \textit{categorical}
(rather than geometric) quotient is Oka,
and to prove the following two weaker properties for $R_3$.

\begin{maindefinition} \label{def:dominable}
Let $X$ be a complex manifold, $p\in X$
and $\phi\colon\cn\to X$ a holomorphic map with $\phi(0)=p$.
(The number $n$ is not necessarily equal to the dimension of $X$.)
We say that $\phi$ \emph{dominates $X$ at $p$}
if $d\phi_0$ is surjective.
If such a $\phi$ exists, then $X$ is
\emph{dominable at $p$},
and $\phi$ is a \emph{dominating map}.
If $X$ is dominable at every $p\in X$,
then $X$ is \emph{strongly dominable}.
\end{maindefinition}

\begin{maindefinition} \label{def:connected}
A manifold $X$ is \emph{strongly $\complex$-connected}
if every pair of points can be joined by an entire curve;
that is, for every pair of points of $X$
there is a holomorphic map $\complex\to X$
whose image contains both points.
\end{maindefinition}

\begin{mainremark}
Every Oka manifold is strongly $\complex$-connected:
this follows from the basic Oka property
described in \cite[page~16]{Forstneric-Larusson-2011}.
The definition of ``$\complex$-connected'' is not standardised:
Gromov in \cite[3.4(B)]{Gromov-1989} uses the term
to refer to strong $\complex$-connectedness as described here,
while other authors use it to refer to the weaker property
that every pair of points can be joined by
a  finite chain of entire curves.
\end{mainremark}

The three main results of this paper
are Theorems~\ref{thm:r3quotient} through \ref{thm:r3connected} below.

\begin{maintheorem} \label{thm:r3quotient}
The categorical quotient for the action of
$\mbsgp\times\mbsgp$ on $R_3$ by pre- and postcomposition
is $\complex$.
\end{maintheorem}

In fact we can give an explicit formula for the quotient map,
as well as a detailed description of the orbits of the group action.

\begin{proof}[Overview of proof]
The quotient map $\pi\colon R_3\to\complex$
is constructed in Section~\ref{section:quotient}.
Given $f\in R_3$, the value of $\pi(f)$
is expressed as a rational function of
cross-ratios of critical points and critical values
of $f$.
Section~\ref{subs:s2} introduces the role of symmetric polynomials
in describing this function;
Section~\ref{subs:standard} describes a standard form
for elements of $R_3$ as an aid to computation,
and Section~\ref{subs:cross_ratio} gives an explicit description of $\pi$
on an open subset of $R_3$.
The extension of $\pi$ to the rest of $R_3$ is given in Section~\ref{subs:quotient}.

To show that $\pi$ is the desired quotient map,
we need to describe the orbits of the group action
and determine which orbits are closed.
This is the content of Section~\ref{subs:orbits}.
Then Lemma~\ref{lemma:same_pi} tells us that $\pi$ distinguishes the closed orbits,
Lemma~\ref{lemma:image_of_pi} and Section~\ref{subs:quotient}
tell us that that the image of $\pi$
is all of $\complex$, and
Corollary~\ref{cor:pi_holo} tells us that $\pi$ is holomorphic.
Remark~\ref{remark:r3quotient} explains why these properties
imply that $\pi$ is the quotient map.
\end{proof}

We can say a little more about the group action:
Section~\ref{subs:stabilisers}
describes the stabilisers of the orbits.
Also, it is interesting to notice
that all the constructions of Section~\ref{section:quotient}
can be carried out within the algebraic category,
whereas the proof of the next two theorems
involves exponential maps.

\begin{maintheorem} \label{thm:r3dominable}
The space $R_3$ of rational maps of degree~3
is strongly dominable.
\end{maintheorem}

\begin{proof}[Overview of proof]
Proposition~\ref{prop:composition}
describes a method of constructing dominating maps,
and the rest of Section~\ref{subs:composition}
gives explicit maps from $\complex^8$ to $R_3$.
The building blocks of this construction
are a map $\eta_0$ from a subset of $\complex^2$ to $R_3$
that is transverse to the orbits of the group action,
a dominating map from $\complex^6$ to the group,
and an embedding of $R_3$ into $\pseven$.
We also need the fact that the domain of $\eta_0$
is dominable: this is Proposition~\ref{prop:chi_surjective}.

Section~\ref{subs:surjective} shows that the image of
$\eta_0$ intersects every orbit of the group action.
This fact allows us to use translates of $\eta_0$
to obtain dominating maps at each point of $R_3$.

Section~\ref{subs:transverse}
contains the proof that $\eta_0$ is transverse to the orbits.
\end{proof}

\begin{maintheorem} \label{thm:r3connected}
The space $R_3$
is strongly $\complex$-connected.
\end{maintheorem}

\begin{proof}
The dominating maps of the previous theorem
are in fact surjective: this follows from Proposition~\ref{prop:all_orbits}.
Thus there exists a surjective map $\complex^8\to R_3$.
Given $f$ and $g$ in $R_3$,
we can choose any preimages of $f$ and $g$,
join the preimages by an entire curve in $\complex^8$,
and compose the entire curve with the surjective map
to obtain an entire curve in $R_3$ joining $f$ and $g$.
\end{proof}

I thank Finnur Lárusson for many helpful discussions
during the preparation of this paper.

\section{Context: rational maps and the parametric Oka~property} \label{section:context}

In Section~\ref{subs:douady} we describe the
complex structure for the spaces $R_d$ of rational maps of degree~$d$.
It is most convenient to use the coefficients of rational maps
as coordinates.
In fact the complex structure thus obtained has an important
universal property,
telling us that it is the right complex structure for our purposes.

In Section~\ref{subs:oka}
we introduce some relevant concepts from Oka theory.
These serve as motivation for the questions addressed in this paper;
however, the definition of an Oka manifold is not directly used.
Consequently, only a brief sketch is given here.
The interested reader can refer to
the survey paper \cite{Forstneric-Larusson-2011}
of Forstnerič and Lárusson for definitions and examples,
or to the book \cite{Forstneric-2011} of Forstnerič
for a more detailed exposition.

\subsection{The space of rational maps} \label{subs:douady}

For each $d=0,1,2\ldots$,
we can embed $R_d$ (as a set) into $\projspace{2d+1}$
by sending a rational function
$$\frac{a_dz^d+a_{d-1}z^{d-1}+\cdots+a_0}{b_dz^d+b_{d-1}z^{d-1}+\cdots+b_0}$$
to the point with homogeneous coordinates
$$(a_d \colon\! a_{d-1}\cdots \colon\! a_0 \colon\! b_d
  \colon\! b_{d-1} \colon\! \cdots \colon\! b_0).$$
We introduce a complex structure on $R_d$ as the pullback of the
complex structure on $\projspace{2d+1}$.
The image of $R_d$ under this embedding is an open subset of $\projspace{2d+1}$.
Specifically, the condition for a rational function $p/q$ to belong to $R_d$,
where $p$ and $q$ are polynomials of maximum degree~$d$,
is that $p$ and $q$ should have no common factors.
This is equivalent to the non-vanishing of the resultant of $p$ and $q$,
and so the image of $R_d$ is the complement of the
resultant locus in $\projspace{2d+1}$.

The topology induced on $R_d$ by this complex structure
coincides with the compact-open topology
on $\scriptO(\pone,\pone)=\bigcup\limits_{d=0}^\infty R_d$.
In this topology, each $R_d$ is connected,
and the map $\scriptO(\pone,\pone)\to\integers$
sending a rational function to its degree is continuous.
Therefore the connected components of $\scriptO(\pone,\pone)$
are precisely the $R_d$.

\begin{proposition}
With the complex structure described above,
the space $\scriptO(\pone,\pone)$
is an internal hom-object in the category
of reduced complex spaces and holomorphic maps.
\end{proposition}

\begin{proof}
We wish to show that $\scriptO(\pone,\pone)$
is a representing object for the functor
$\scriptO(- \times \pone, \pone)$.

Given a reduced complex space $T$
and a holomorphic map $\phi \colon T\times\pone\to\pone$,
define $\tilde\phi \colon T\to\scriptO(\pone,\pone)$
by $\tilde\phi(t)(x)=\phi(t,x)$.
For each $t$,
the map $\tilde\phi(t)$
is a member of some $R_d$;
by continuity of the degree map,
the degree~$d$ must be constant on each connected component of $T$.
Thus we can use the embedding $R_d\to\projspace{2d+1}$
to write $\tilde\phi$ in local coordinates.
From this, it is easy to see that $\tilde\phi$ is holomorphic.

Let $\eta_T \colon \scriptO(T\times\pone,\pone)\to\scriptO(T,\scriptO(\pone,\pone))$
be the map sending $\phi$ to $\tilde\phi$.
Then $\eta$ is natural transformation
from the functor
$\scriptO(- \times X, Y)$ to the functor $\scriptO(-,\scriptO(X,Y))$.

If $\psi \colon T\to\scriptO(\pone,\pone)$ is holomorphic,
then $\psi=\tilde\phi$ where $\phi(t,x)=\psi(t)(x)$,
and it is immediate that $\phi$ is holomorphic.
Therefore $\eta_T$ is a bijection,
and $\eta$ is a natural isomorphism.
\end{proof}

This result is implicit in the work
of Kaup~\cite{WKaup-1969}, but is not stated explicitly there.
See also Douady~\cite[Section~10.2]{Douady-1966}
regarding universal properties of mapping spaces.

\subsection{A brief outline of Oka theory} \label{subs:oka}

Informally speaking,
Oka manifolds can be viewed as the opposite
of Kobayashi hyperbolic manifolds.
Hyperbolicity is a type of holomorphic rigidity property;
conversely, Oka manifolds enjoy a variety of holomorphic flexibility properties.

The most concrete expression of flexibility for a manifold $X$
is that there should be ``many'' holomorphic maps $\complex\to X$.
This is formalised in Gromov's notion of ellipticity,
introduced in~\cite{Gromov-1989}.
Dominability and $\complex$-connectedness (Definitions~\ref{def:dominable}
and~\ref{def:connected}) express weaker versions of the same idea.
There is a chain of implications
$$\text{elliptic}\Rightarrow\text{Oka}\Rightarrow\text{strongly dominable
and strongly }\complex\text{-connected};$$
at present, it is unknown whether the reverse implications hold
in general.
(Campana and Winkelmann~\cite[Example~8.3]{Campana-Winkelmann-2012}
give an example of a manifold that is $\complex$-connected but not Oka.
However, there are no known examples of manifolds
that are strongly dominable but not Oka.)

Oka manifolds also enjoy a number of homotopy properties.
The simplest is the so-called basic Oka property (BOP):
every continuous map from a Stein manifold to an Oka manifold
is homotopic to a holomorphic map.
Oka manifolds satisfy a stronger version of the BOP
with added approximation and interpolation conditions
(see \cite[page~16]{Forstneric-Larusson-2011} for details).

Forstnerič and Lárusson have identified a number of equivalent
properties which characterise Oka manifolds.
The following is of particular interest
in the context of mapping spaces.

\begin{definition}[Parametric Oka property (POP), simple version]
A manifold $X$ satisfies the \emph{parametric Oka property}
if for every Stein manifold $S$ and every compact subset
$P\subset\reals^m$,
every continuous map $f \colon P\times S\to X$
is homotopic to a map
$f_1 \colon P\times S\to X$
such that $f_1(\cdot,x) \colon S\to X$
is holomorphic for every $x\in P$.
\end{definition}

In other words, a family of continuous maps
can be deformed to a family of holomorphic maps
with continuous dependence on the parameter.

This is apparently a stronger condition than the BOP.
However, it turns out that the POP with approximation and interpolation
is equivalent to the BOP with approximation and interpolation;
either condition can be taken as the definition of an Oka manifold
(see \cite[Section~1]{Forstneric-2009}).

It is natural to ask whether continuous dependence on a parameter
can be replaced by holomorphic dependence.
To put it another way, if $P$ is a compact complex manifold,
then is the mapping space $\scriptO(P,X)$
an Oka manifold or similar?
(The results of \cite{Douady-1966}
guarantee that if $P$ is compact,
then $\scriptO(P,X)$ carries a universal complex structure.)
In this paper we begin with the simplest interesting case,
that of $P=X=\pone$.

\section{Group actions and the degree 2 case} \label{section:actions}

The action of the Möbius group $\mbsgp$
on the Riemann sphere $\rsphere$ is sharply 3-transitive:
given any two triples of distinct points of $\rsphere$,
there is a unique Möbius transformation taking the first triple to the second.
Because of this transitivity,
the Möbius group is a valuable tool
for simplifying the study of $R_d$ when $d$ is small.
Specifically, rational maps can be composed with Möbius transformations,
giving rise to a number of interesting group actions on $R_d$.
Since $\mbsgp$ is reductive
(being the complexification of the compact subgroup $\PSU_2$),
we can study its actions using
geometric invariant theory as
described in~\cite{Snow-1982}
(which gives analytic analogues of the
results of~\cite{Luna-1973}).

First there are the pre- and postcomposition actions.
Precomposition is the action $R_d\times\mbsgp\to R_d$
defined by
$f\cdot g=f\circ g$ for $f\in R_d$ and $g\in\mbsgp$.
Postcomposition is the action $\mbsgp\times R_d\to R_d$
defined by $g\cdot f=g\circ f$.
An interesting asymmetry appears here.
It is easy to verify that postcomposition is a free action.
Therefore there is a well defined geometric quotient:
the set of orbits has the structure of a complex manifold,
and the quotient map is a fibre bundle~\cite[Corollary~5.5]{Snow-1982}.
This quotient space is a useful tool
in studying the topology of $R_d$:
see for example \cite{Havlicek-1995} and \cite{Ono-Yamaguchi-2003}.
On the other hand,
the precomposition action is not free.
(For example, consider $f(x)=x^d$ and let $g$ be
multiplication by a $d$th root of unity.)
This group action has received much less attention in the literature.

The two group actions can be combined to give the conjugation action:
$f\cdot g=g^{-1}\circ f\circ g$.
This action is of interest in the study of holomorphic dynamics:
see for example \cite[Section~3]{Milnor-1993}.
In particular, the quotient space of $R_2$ under this action
is $\complex^2$
(see also \cite[Section~5]{Silverman-1998}).
For $d>2$, the quotient of $R_d$
is a rational variety~\cite[Section~4]{Levy-2011}.
However, the behaviour of holomorphic flexibility properties
under birational maps is not well understood,
so it is not obvious how to apply this group action
to our present investigations.

The actions can also be considered jointly via the two-sided action
of $G=\mbsgp\times\mbsgp$ on $R_d$:
if $g=(g_1,g_2)\in G$ and $f\in R_d$,
then define $f^g$ by
$$f^g=g_1^{-1}\circ f\circ g_2.$$

The two-sided action, like the precomposition action,
is not free.
Furthermore, for $d\geq3$
there exist non-closed orbits (see Section~\ref{subs:orbits} below),
so the orbit space is not Hausdorff;
the geometric quotient does not exist as a manifold.
However, Snow's main theorem
tells us that the categorical quotient for this action
exists as a reduced complex space.
Since $R_3$ is $7$-dimensional and $G$ is $6$-dimensional,
we expect to find a $1$-dimensional quotient space.
Study of this quotient reveals a great deal about the structure of $R_3$.
This will be explored further
in Section~\ref{section:quotient} below.

For the case $d=2$, the action is transitive.
This can be proved by elementary means,
as in \cite{Guest-et-al-1995}.
An alternative and more intuitive proof
can be obtained by considering rational maps in terms of critical values.

A rational map of degree $d$ can be viewed as a branched
$d$-sheeted covering map $\rsphere\to\rsphere$.
By the Riemann--Hurwitz formula, there are $2d-2$ critical values
when counted with multiplicity.
Each critical value has multiplicity at most $d-1$,
so there must be at least two distinct critical values.

The Riemann existence theorem
(see for example~\cite[page~49]{Donaldson-2011})
plays a key role in understanding the orbits.
We need only the following special case.

\begin{maintheorem}[Riemann existence theorem, special case]\label{thm:RET}
Let $\Delta$ be a finite subset of $\pone$,
and $\phi \colon \pi_1(\pone\setminus\Delta)\to\Sym{d}$
a group homomorphism
(where $\Sym{d}$ denotes the symmetric group on $d$ symbols).
Suppose the image of $\phi$ is transitive.
Then there exists a compact connected Riemann surface $X$
and a $d$-fold branched holomorphic covering map
$f \colon X\to\pone$
with critical values $\Delta$ and monodromy given by $\phi$.
If $f_1 \colon X_1\to\pone$ and $f_2 \colon X_2\to\pone$ are two such coverings,
then there exists a biholomorphic map $g \colon X_1\to X_2$
such that $g\circ f_1=f_2$.
\end{maintheorem}

The genus of the surface $X$ is given by the Riemann--Hurwitz formula;
the multiplicities of the critical values
can be calculated from the cycle structure of the permutations
as described in Chapter~1 of \cite{Lando-Zwonkin-2004}.

If the multiplicities sum to $2d-2$,
then the resulting covering map is a rational function $\pone\to\pone$,
and the orbit of this rational function under the precomposition
action described above
is uniquely determined by the critical values and monodromy.

In the case of $d=2$, there are exactly two distinct critical values.
For a two-sheeted covering,
there is only one possible monodromy permutation.
Triple transitivity of the group implies
that pairs of critical values are all equivalent under the postcomposition action.
It follows that the two-sided action on $R_2$
has only one orbit.
This proves Theorem~\ref{thm:r2oka}.

\section{Degree 3: the categorical quotient} \label{section:quotient}

In this section we study the two-sided group action
of $G=\mbsgp\times\mbsgp$ on $R_3$
from the point of view of geometric invariant theory.
The group is reductive,
and $R_3$, being the complement of a hypersurface in $\pseven$,
is Stein.
Therefore the main theorem of Snow's paper \cite{Snow-1982}
tells us that the categorical quotient exists
and is a reduced Stein space.
We will show that the quotient is in fact $\complex$,
proving Theorem~\ref{thm:r3quotient} above.
The quotient map is explicitly described
in Section~\ref{subs:cross_ratio},
and the proof is completed in Section~\ref{subs:quotient}.

A rational function of degree~3 has four critical points
and four critical values, counted with multiplicity.
The multiplicity of each critical value is either one or two.
Therefore there are only three possible cases
(all of which occur):
\begin{itemize}
  \item[-] four distinct simple critical values;
  \item[-] one double and two simple critical values;
  \item[-] two double critical values.
\end{itemize}
The set of rational functions with four distinct critical values
is an open subset of $R_3$,
and will be referred to as the \emph{open stratum},
denoted $R_3^O$.
The open stratum is the set of $f\in R_3$
such that the zeros of $f'$ are distinct;
in other words, in local coordinates
it is the complement of the zero locus
of the discriminant of the numerator of $f'$.
Therefore it is an open subset of $R_3$,
in both the compact-open topology
and the Zariski topology.

The complement of $R_3^O$ will be called the
\emph{null fibre},
for reasons that will become clear later.
Most of this section will be concerned with understanding the
orbits in the open stratum.
For each of the other two cases there is a single orbit;
this is proved in Section~\ref{subs:orbits} below.

\begin{mainremark} \label{remark:dimensions}
Since $R_3$ is $7$-dimensional and $G$ is $6$-dimensional,
it is reasonable to expect that the generic orbit will have codimension~1.
In fact, if $g=(g_1,g_2)\in G$ fixes $f\in R_3$,
then $g_1$ must permute the critical values of $f$,
and $g_2$ permutes the critical points.
Since an element of the Möbius group is uniquely determined
by the image of three points,
it follows that if $f$ has at least three critical values
(and therefore at least three critical points),
then the stabiliser of $f$ is finite,
so the orbit is $6$-dimensional.

In the case where $f$ has only two critical values,
the stabiliser can contain a one-parameter subgroup.
For example, $f=x^3$
is fixed by the group element $(a^3x,ax)$
for all $a\in\cstar=\complex\setminus\{0\}$.
Hence the orbit of such $f$ is at most $5$-dimensional.
\end{mainremark}

More details about the stabilisers are given
in Section~\ref{subs:stabilisers}.

\begin{mainremark} \label{remark:null_codim}
The null fibre, being locally
the zero locus of a discriminant polynomial,
is a proper analytic subvariety of~$R_3$.
Since it contains at least one $6$-dimensional orbit
(the non-closed orbit of Proposition~\ref{prop:orbits}),
it has codimension~1.
\end{mainremark}

\subsection{Cross-ratio and symmetrised cross-ratio} \label{subs:s2}

Given four distinct points (a \emph{quartet})  $z_1,z_2,z_3,z_4\in\complex$,
their cross-ratio is the number
$$(z_1,z_2;z_3,z_4)=\frac{(z_1-z_3)(z_2-z_4)}{(z_2-z_3)(z_1-z_4)}.$$
The definition is extended to quartets in $\pone$
by adopting the convention that $\infty/\infty=1$.
For example,
$(0,\infty;1,\lambda)=(-1\cdot\infty)/(\infty\cdot(-\lambda))=\lambda$.
For quartets of distinct points, the cross-ratio can take on any value except
0, 1 or~$\infty$.
The Möbius group $\mbsgp$ preserves cross-ratio.
Therefore we aim to use the cross-ratio to construct invariant functions
for the action of $G=\mbsgp\times\mbsgp$ on $R_3$.
In particular, we are interested in
the cross-ratios of the critical points
and of the critical values of an element of the open stratum.

There is a technical issue that needs to be addressed:
there is no canonical way of ordering the four critical points or values.
Therefore the ``cross-ratio of the critical points'' is not well defined.
We will address this by symmetrising the cross-ratio.
(This is analogous to the relationship between
the elliptic modular function $\lambda$
and Klein's $j$-invariant
given by $j(\tau)=256(1-\lambda+\lambda^2)^2/\lambda^2(1-\lambda^2)$;
the $j$-invariant is a symmetrised version of the modular function.)

Generically, the 24 possible orders of four points
give rise to six cross-ratios.
If one ratio is $\lambda$, then the six ratios are
\begin{equation}\label{eq:sixXratios}
\lambda,\frac{1}{\lambda},1-\lambda,\frac{1}{1-\lambda},
\frac{\lambda}{\lambda-1},\frac{\lambda-1}{\lambda}.
\end{equation}
For most values of $\lambda$ these six numbers are distinct.
If $\lambda$ is one of $-1$, $\tfrac{1}{2}$ or~2,
then there are only three distinct cross-ratios, namely
$\{-1,\tfrac{1}{2},2\}=\{\lambda,1/\lambda,1-\lambda\}$.
If $\lambda$ is a primitive sixth root of unity,
i.e.\ $\lambda=e^{\pm\pi i/3}$,
then there are only two distinct cross-ratios,
namely $\lambda$ and $\bar\lambda$.

We will write
$\sigma_1,\ldots,\sigma_6$
for the elementary symmetric functions of six variables.
Thus $\sigma_1(x_1,\ldots,x_6)=x_1+\cdots+x_6$,
$\sigma_2(x_1,\ldots,x_6)=x_1x_2+\cdots+x_5x_6$ (fifteen terms),
and so on up to
$\sigma_6(x_1,\ldots,x_6)=x_1x_2x_3x_4x_5x_6$.

For $k=1,\ldots,6$, define functions
$s_k \colon \complex\setminus\{0,1\}\to\complex$
by
$$s_k(\lambda)=\sigma_k\left(\lambda,\frac{1}{\lambda},1-\lambda,\frac{1}{1-\lambda},
\frac{\lambda}{\lambda-1},\frac{\lambda-1}{\lambda}\right).$$
It follows that for a quartet $(z_1,z_2,z_3,z_4)$ of distinct points of $\pone$,
the quantity $s_k((z_1,z_2;z_3,z_4))$
depends only on the set $\{z_1,z_2,z_3,z_4\}$.
Thus we can regard these quantities
as cross-ratios of an unordered set.

Routine calculations (easily verified using a computer algebra system:
see Appendix~\ref{appendix:sage})
show that $s_k$ is constant when $k=1$, $5$ or $6$:
we have $s_1=s_5=3$ and $s_6=1$.
Also $s_4=s_2$ and $s_3=2s_2-5$.
We will only use $s_2$ in the sequel.

It is also worth noting that the function $s_2+3/4$ factorises nicely.
Thus we define $s:\complex\setminus\{0,1\}\to\complex$ by
\begin{equation} \label{eq:symmXratio}
s(\lambda)=s_2(\lambda)+\frac{3}{4}
=-\frac{(\lambda+1)^2(2\lambda-1)^2(\lambda-2)^2}{4\lambda^2(\lambda-1)^2},
\end{equation}
and we say that the
\emph{symmetrised cross-ratio}
of a set of four distinct points $\{z_1,z_2,z_3,z_4\}$
of $\pone$ is the number
$s((z_1,z_2;z_3,z_4))$.

\begin{proposition} \label{prop:symmXratio}
The symmetrised cross-ratio
is a complete invariant
for the action of the Möbius group
on unordered sets of four distinct points of $\pone$.
\end{proposition}

\begin{proof}
The (usual) cross-ratio is a complete invariant
for the action on ordered sets of four points.
Changing the order of the points
transforms the cross-ratio as described above.
Therefore we simply need to show that if
$s(\mu)=s(\lambda)$,
then $\mu$ is one of the six quantities listed above
at \eqref{eq:sixXratios}.
This follows from the following identity,
easily verified by mechanical calculation:
$$
    \mu^2(\mu-1)^2(s_2(\lambda)-s_2(\mu))
    = (\mu-\lambda)(\mu-\frac{1}{\lambda})\cdots(\mu-\frac{\lambda-1}{\lambda}).
    \qedhere
$$
\end{proof}

Using this symmetrised cross-ratio,
we define a map $R_3^O\to\complex$, also called $s$, as follows.
For $f\in R_3^O$ with critical points $z_1,z_2,z_3,z_4$,
\begin{equation} \label{eq:s}
s(f)= s((z_1,z_2;z_3,z_4)).
\end{equation}
The dependence of the critical points on $f$ is continuous
by Hurwitz's theorem,
and so $s$ is continuous.

\subsection{Closed and non-closed orbits} \label{subs:orbits}

We can think of the categorical quotient
as parametrising the closed orbits of the group action.
Therefore we need to determine which orbits are closed.

First we will deal with the null fibre,
i.e.\ the set of functions in $R_3$
whose critical values are not distinct.
There are two cases.

Suppose $f$ has exactly three critical values:
one double and two simple.
By the transitivity of the postcomposition action,
we see that $f$ lies in the same orbit as a function $g$
with $\infty$ as a double critical value, i.e.\ a polynomial.
We can also assume that 0 is a simple critical value;
and by transitivity of the precomposition action
we can require $g(0)=0$ and $g(\infty)=\infty$.
Such $g$ must be of the form $ax^3+bx^2$
for some $a,b\in\cstar$.
Conversely, every such function has exactly three critical values:
$\infty$ is a double critical value of every cubic polynomial,
and the finite critical values are the two roots
$0$ and $-2b/3a$ of $g'$.

If $g(x)=ax^3+bx^2$, then
$$\frac{a^2}{b^3}g\left(\frac{b}{a}x\right)
=\frac{a^2}{b^3}\left(a\frac{b^3}{a^3}x^3+b\frac{b^2}{a^2}x^2\right)
=x^3+x^2,$$
so $g$, and therefore $f$, is in the same orbit as $x^3+x^2$.
Thus the set of functions with exactly three critical values is a single orbit.

It is easy to see that this orbit is not closed:
the sequence $(x^3+\tfrac{1}{n}x^2)_{n=1}^{\infty}$
converges to $x^3$, which is outside the orbit
because it has only two critical values.
This fact has a geometrical interpretation:
travelling along the sequence, the two simple critical values
of $(x^3+\tfrac{1}{n}x^2)$ get closer and eventually coalesce.

The second case is that of functions with two double critical values.
This is the smallest possible number of critical values
for an element of $R_3$.
Bearing in mind the above geometrical interpretation,
it is immediate that having two critical values is a closed condition:
it is not possible for the critical values to coalesce within $R_3$.
Thus the set of such functions is closed.
Similar arguments to those presented above
show that all such functions lie in the same orbit as $x^3$.

To summarise:
the null fibre consists of exactly two orbits, namely
a non-closed orbit consisting of the functions with exactly three critical values,
and a closed orbit consisting of the functions with exactly two critical values.

Recall (Remark~\ref{remark:dimensions})
that orbits of functions with at least three critical values
are $6$-dimensional, and that any other orbits are of strictly smaller dimension.
It follows that the closed orbit in the null fibre
is the unique orbit in $R_3$ of minimal dimension.
We will see in Theorem~\ref{thm:stabilisers}
that this orbit is in fact $5$-dimensional.

Now we turn our attention to the open stratum $R_3^O$,
i.e.\ the set of functions in $R_3$
with four distinct critical values.
Here we use the map $s \colon R_3^O\to\complex$ defined in the previous section.
Since the group action preserves cross-ratio,
each fibre of $s$ is a union of orbits.

Suppose the orbit $f^G$ of $f\in R_3^O$ is not closed.
Since $s$ is continuous,
the closure of $f^G$ is contained in $s^{-1}(s(f))$.
Proposition~2.3 of \cite{Snow-1982}
tells us that the closure of $f^G$
contains an orbit of strictly smaller dimension,
and therefore $s^{-1}(s(f))$ contains orbits of at least two different dimensions.
But since all orbits in the open stratum are $6$-dimensional,
this is impossible.

The following proposition collects together the results obtained so far.

\begin{proposition}\label{prop:orbits}
The orbits for the action of $G=\mbsgp\times\mbsgp$ on $R_3$ are of three types.
\begin{itemize}
  \item[-] The points with two distinct critical values form a single orbit,
    which is closed and $5$-dimensional.
  \item[-]The points with three distinct critical values form a single orbit,
    which is non-closed and $6$-dimensional.
    The closure of this orbit is its union with the $5$-dimensional orbit.
  \item[-]The orbit of a function with four distinct critical values
    is closed and $6$-dimensional.
\end{itemize}
\end{proposition}

\subsection{Standard form for cubic rational functions} \label{subs:standard}

The following will be useful as an aid to calculation.

\begin{definition} \label{def:standard_form}
Let $a\in\complex$. Then $f_a$ will denote the rational function
$$f_a(x)=\frac{x^2(x+a)}{(2a+3)x-(a+2)}.$$
\end{definition}

\begin{remark} \label{remark:fa}
The function $f_a$ fixes the points
$0$, $1$ and~$\infty$.
If $a$ is $-1$ or $-2$,
then $f_a$ equals $x^2$ or $x(2-x)$ respectively.
Otherwise $f_a\in R_3$,
and $0$, $1$ and $\infty$ are critical points and critical values of~$f_a$.
\end{remark}

\begin{lemma} \label{lemma:standard_form}
Suppose $f\in R_3^O$ has critical points
$0$, $1$, $\infty$ and~$\mu$
and critical values
$0$, $1$, $\infty$ and~$\lambda$,
with $f$ sending $0$, $1$, $\infty$ and~$\mu$
to $0$, $1$, $\infty$ and~$\lambda$ respectively.
Then there exists unique $a\in\complex\setminus\{0,-1,-3/2,-2,-3\}$ such that
$f=f_a$.
The values of $a$, $\mu$ and~$\lambda$
are related by the equations
\begin{equation}\mu=-\frac{a(a+2)}{2a+3},\quad
\lambda=\frac{\mu^3}{(a+2)^2},\quad
a=\frac{\mu^3+3\mu\lambda-4\lambda}{2\lambda(1-\mu)}.
\label{eq:a_from_lambda}
\end{equation}
Conversely, $f_a\in R_3^O$
for all $a\in\complex\setminus\{0,-1,-3/2,-2,-3\}$.
\end{lemma}

\begin{remark}
Given an arbitrary element of $R_3^O$,
we can pre- and postcompose with Möbius transformations
to send three of the critical points and values
to $\{0,1,\infty\}$.
Thus every orbit in $R_3^O$ contains at least one $f_a$.
\end{remark}

\begin{proof}[Proof of lemma]
For 0 to be a double zero of $f$ but not a triple zero,
the numerator of $f$ must take the form $cx^2(x+a)$
for some $a,c\in\cstar$.
Since $f(1)\neq0$, we have the condition $a\neq -1$.
For $\infty$ to map to $\infty$ with multiplicity exactly~2,
the denominator must be a linear polynomial,
say $x+b$ for some $b\in\cstar\setminus\{a\}$.
(We can take the leading coefficient to be 1 because
we have the coefficient $c$ in the numerator.
We require $b\notin\{0,a\}$ in order for $f$ to have degree~3.)
Thus $f$ is of the form
$$f(x)=\frac{cx^2(x+a)}{x+b}.$$
The condition $f(1)=1$ gives
$$\frac{c(1+a)}{1+b}=1,$$
so $b=c(a+1)-1$, and therefore
$$f(x)=\frac{cx^2(x+a)}{x+(a+1)c-1}.$$

The finite critical points are exactly the zeros of $f'$.
These zeros must be $0$, $1$ and~$\mu$.
By the quotient rule,
the numerator of $f'$ is
\begin{equation} \label{eq:cvalue}
cx(3x+2a)(x+(a+1)c-1)-cx^2(x+a).
\end{equation}
Evaluating this at $x=1$ gives
$$c^2(3+2a)(a+1)-c(1+a)=(a+1)c((2a+3)c-1),$$
which vanishes when $a=-1$ or $c=0$
or $c=1/(2a+3)$.
The first two cases are impossible.
Substituting the third value of $c$ into the above expression for $f$,
and dividing the numerator and denominator by $c$, gives
$$f(x)=\frac{x^2(x+a)}{(2a+3)x+a+1-(2a+3)}=\frac{x^2(x+a)}{(2a+3)x-(a+2)},$$
as required.

Now we wish to calculate the values of $\mu$ and $\lambda$.
We have already ensured that 0, 1 and $\infty$ are critical points of $f$;
the fourth critical point is $\mu$.
Substituting the value of $c$ into \eqref{eq:cvalue}
and dividing by the common factor
$cx$ gives
\begin{align*}
\noalign{$(3x+2a)(x+(a+1)c-1)-x(x+a)$}
  &=(3x+2a)(x+\tfrac{a+1}{2a+3}-1)-x(x+a) \\
  &= \tfrac{1}{2a+3}((3x+2a)((2a+3)x+a+1-(2a+3))-(2a+3)x(x+a)) \\
  &= \tfrac{1}{2a+3}((4a+6)x^2+(2a^2-6)x-2a(a+2)) \\
  &= \tfrac{2}{2a+3}((2a+3)x^2+(a^2-3)x-a(a+2)) \\
  &= \tfrac{2}{2a+3}(x-1)((2a+3)x+a(a+2)).
\end{align*}
This vanishes at $x=1$ (which we already know to be a critical point)
and at $x=\mu=-a(a+2)/(2a+3)$.
Then $\lambda$ is given by $f(\mu)$; multiplying numerator and denominator
by $(2a+3)^3$ we obtain
\begin{align*}
f(\mu)
  &=\frac{a^2(a+2)^2(-a(a+2)+a(2a+3))}{-a(a+2)(2a+3)^2-(a+2)(2a+3)^3} \\
  &=\frac{a^2(a+2)^2(a^2+a)}{-(2a+3)^2(a(a+2)+(a+2))} \\
  &=\frac{-a^3(a+2)^2(a+1)}{(2a+3)^2(a+1)(a+2)} \\
  &=\frac{-a^3(a+2)}{(2a+3)^3} \\ 
  &=\frac{\mu^3}{(a+2)^2},
\end{align*}
as required.

Conversely, $a$ can be calculated from $\mu$ and $\lambda$ as follows.
The second equation of the lemma
can be rearranged to give
\begin{align}
  (2a+3)\mu &= -a(a+2), \notag \\
\shortintertext{and so}
  a^2+2(\mu+1)a+3\mu &= 0. \label {eq:a_mu_quadratic} \\
\shortintertext{The third equation from Lemma~\ref{lemma:standard_form} gives}
  a^2+4a+4 &= \mu^3/\lambda. \label{eq:a_mu_lambda}
\shortintertext{Subtracting \eqref{eq:a_mu_quadratic}
from \eqref{eq:a_mu_lambda}:}
  2(1-\mu)a+4-3\mu &= \mu^3/\lambda, \notag
\end{align}
which gives the required expression for $a$.

Next, we need to identify the ``forbidden'' values of $a$.
These come from the constraints $\mu,\lambda\notin\{0,1,\infty\}$.
We have $\mu=0$ exactly when $a=0$ or $a=-2$,
and $\mu=\infty$ when $a=-3/2$.
The equation $\mu=1$ gives $a(a+2)+2a+3=0$,
which factorises as $a^2+4a+3=(a+1)(a+3)=0$,
eliminating the values $a=-1,-3$.
Looking at $\lambda=0$ and $\lambda=\infty$ gives nothing new.
The equation $\lambda=1$ gives
\begin{align*}
0 &=\lambda-1 \\
  &=\mu^3-(a+2)^2 \qquad\text{(if $a\neq-2$)}\\
  &=(2a+3)^3(\mu^3-(a+2)^2) \qquad\text{(if $a\neq-3/2$)}\\
  &=-a^3(a+2)^3-(a+2)^2(2a+3)^3 \\
  &=-(a+2)^2(a^3(a+2)+(2a+3)^3) \\
  &=-(a+2)^2(a^4+2a^3+8a^3+36a^2+54a+27) \\
  &=-(a+2)(a+1)(a^3+9a^2+27a+27)\\
  &=-(a+1)(a+2)(a+3)^3,
\end{align*}
so again no new forbidden values are obtained.

Finally, if $a$ is not one of the forbidden values,
then it is clear that we can form the function
$f_a$ of Definition~\ref{def:standard_form},
that it is an element of $R_3$,
and that the corresponding values of $\mu$ and $\lambda$
are not in $\{0,1,\infty\}$, so that $f_a\in R_3^O$.
Thus every value of $a$ in $\complex\setminus\{0,-1,-3/2,-2,-3\}$ can be realised.
\end{proof}

\begin{remark} \label{remark:special_value}
Given $\mu$, equation~\eqref{eq:a_mu_quadratic} in general
gives two possible values of $a$,
and therefore two possible values of $\lambda$.
The discriminant of \eqref{eq:a_mu_quadratic} is
$$\Delta=(2(\mu+1))^2-12\mu=4(\mu^2-\mu+1).$$
Therefore there is a unique value of $a$ exactly when
$\mu=e^{\pm\pi i/3}$;
as mentioned in Section~\ref{subs:s2},
these are the cross-ratio values for which different orderings of the critical points
give only two distinct cross-ratios
rather than the usual six.
\end{remark}

\subsection{Cross-ratio and invariant functions} \label{subs:cross_ratio}

Our goal is to find a complete set of invariants
for the action of $G$.
The function $s$ of \eqref{eq:s}
is invariant,
but we will see in Example~\ref{ex:a_values}
that $s$ is not sufficient to distinguish the closed orbits.

In this section we define a new function $\pi$,
described in \eqref{eq:pi} below,
using the results of the previous section.
We will see that this function is in fact the categorical quotient map.
The definition parallels that of $s$:
the quantity $a$ of Lemma~\ref{lemma:standard_form}
plays the role of the cross-ratio,
Lemma~\ref{lemma:a_transform} plays the role of \eqref{eq:sixXratios},
and the elementary symmetric function $\sigma_2$ is again used.

Let $f\in R_3^O$.
Choose an ordering $\sigma$ of the critical values,
and let $\lambda$ be the cross-ratio
of the critical values in that order.
Each critical value has two preimages,
one of which is a critical point,
so there is an induced ordering of the critical points.
Let $\mu$ be the cross-ratio of the critical points in this order.

\begin{definition} \label{def:signature}
The \emph{signature} of $f$
with respect to $\sigma$
is the pair $(\mu,\lambda)$.
\end{definition}

\begin{lemma} \label{lemma:signature}
Two elements of $R_3^O$ are in the same orbit
if and only if there exist orderings for which they have the same signature.
\end{lemma}

\begin{proof}
Let $f\in R_3^O$ and choose an ordering $\sigma$ of its critical values.
If we precompose $f$ with a Möbius transformation,
then the critical points move but the critical values are unchanged.
Similarly, postcomposition will move the critical values
but leave the critical points unchanged.
Given $\sigma$, there is a unique Möbius transformation $\alpha_1$ moving the first
three critical points to 0, $\infty$ and~1 in order.
Since cross-ratio is preserved, the fourth critical point will be moved to $\mu$.
Similarly, there is a unique Möbius transformation $\alpha_2$
moving the critical values to 0, $\infty$, 1 and~$\lambda$ in order.
Write $f^{(\sigma)}$ for the function $\alpha_2\circ f \circ \alpha_1^{-1}$.
Note that $f^{(\sigma)}$ has critical points 0, 1, $\infty$ and~$\mu$,
critical values 0, 1, $\infty$ and~$\lambda$,
and fixes the points 0, 1 and~$\infty$.

It follows that $f^{(\sigma)}$
is in fact the function $f_a$ of
Definition~\ref{def:standard_form},
for the value of $a$ given by \eqref{eq:a_from_lambda}.
Hence $f^{(\sigma)}$ is uniquely determined
by the signature.

If $f$ and $g$ have the same signature with respect to
orderings $\sigma$, $\rho$,
then $f^{(\sigma)}=g^{(\rho)}$,
and hence $f$ and~$g$ are in the same orbit.
Conversely, suppose $f$ and~$g$ are in the same orbit,
and choose an ordering $\sigma$ for $f$.
Then there exist Möbius transformations $\beta_1$ and $\beta_2$
such that $\beta_2\circ g\circ\beta_1^{-1}=f^{(\sigma)}$.
Taking the critical values of $g$ in the ordering $\rho$ given by
$\beta^{-1}(0)$, $\beta^{-1}(1)$, $\beta^{-1}(\infty)$, $\beta^{-1}(\mu)$,
we see that $g^{(\rho)}=f^{(\sigma)}$.
\end{proof}

We would like to use the quantity $a$ of Lemma~\ref{lemma:standard_form}
to parametrise the orbits.
However, an orbit can contain more than one $f_a$.
The situation is analogous to that of a quartet of points
having more than one cross-ratio,
and we resolve it in the same way, by symmetrising
with respect to the set of values that can occur.

\begin{lemma} \label{lemma:a_transform}
Let $a,b\in\complex\setminus\{0,-1,-3/2,-2,-3\}$.
Then $f_a$ and $f_b$ are in the same orbit
if and only if $b$ is in the set
$$
  \left\{
      a,\>-\frac{2a+3}{a+2},\> -(a+3),\> -\frac{a}{a+1},\>
      -\frac{2a+3}{a+1},\> -\frac{a+3}{a+2}
  \right\}.$$
\end{lemma}

\begin{proof}
By Remark~\ref{remark:fa},
$f_a$ has critical points $0$, $1$, $\infty$ and $\mu$,
and critical values $0$, $1$, $\infty$ and $\lambda$,
for some $\mu,\lambda\in\pone$.
Therefore $f_a$ has signature $(\mu,\lambda)$.
Since $a\not\in\{0,-1,-3/2,-2,-3\}$,
Lemma~\ref{lemma:standard_form} implies that
$f_a\in R_3^O$, so
$\mu,\lambda\in\complex\setminus\{0,1\}$,
and 
$$a=\frac{\mu^3+3\mu\lambda-4\lambda}{2\lambda(1-\mu)}.$$
Similarly, let $\mu'$ and $\lambda'$ be the fourth critical point
and critical value respectively of $f_b$, so that
$f_b$ has signature $(\mu',\lambda')$
and
$$b=\frac{\mu'^3+3\mu'\lambda'-4\lambda'}{2\lambda'(1-\mu')}.$$
If the critical points and critical values of $f_a$
are taken in a different order,
then the cross-ratios change as described in \eqref{eq:sixXratios}.
Thus the signatures of $f_a$ with respect to the
various orderings are
$$
\left(\mu,\lambda\right), 
\left(\tfrac{1}{\mu}, \tfrac{1}{\lambda}\right),
\left(1-\mu, 1-\lambda\right),
\left(\tfrac{1}{1-\mu}, \tfrac{1}{1-\lambda}\right),
\left(\tfrac{\mu}{\mu-1}, \tfrac{\lambda}{\lambda-1}\right),
\left(\tfrac{\mu-1}{\mu}, \tfrac{\lambda-1}{\lambda}\right).
$$
It follows from Lemma~\ref{lemma:signature}
that $f_a$ and $f_b$ are in the same orbit
if and only if $(\mu',\lambda')$
equals one of the six pairs listed above.
We will show that $(\mu',\lambda')=(1/\mu, 1/\lambda)$
if and only if $b=-(2a+3)/(a+2)$.
The other cases are handled similarly.

First suppose that $\mu'=1/\mu$ and $\lambda'=1/\lambda$.
Using \eqref{eq:a_from_lambda},
we have
\begin{align*}
b &= \frac{\mu^{-3}+3\mu^{-1}\lambda^{-1}-4\lambda^{-1}}{2\lambda^{-1}(1-\mu^{-1})} \\
  &= \frac{\mu^{-3}\lambda+3\mu^{-1}-4}{2(1-\mu^{-1})} \\
  &=\frac{(a+2)^{-2}-3(2a+3)a^{-1}(a+2)^{-1}-4}
         {2(1+(2a+3)a^{-1}(a+1)^{-1})} \\
  &=\frac{a-3(2a+3)(a+2)-4a(a+2)^2}
         {2(a(a+2)^2+(2a+3)(a+2))} \\
  &=\frac{a-3(2a^2+7a+6)-4a(a^2+4a+4)}
         {2(a+2)(a^2+2a+2a+3)} \\
  &=\frac{-4a^3-22a^2-36a-18}
         {2(a+2)(a^2+4a+3)} \\
  &=-\frac{2a^3+11a^2+18a+9}
         {(a+2)(a+1)(a+3)} \\
  &=-\frac{(a+1)(2a+3)(a+3)}
         {(a+1)(a+2)(a+3)} \\
  &=-\frac{2a+3}
         {a+2}.
\intertext{Conversely, if $b=-(2a+3)/(a+2)$, then}
\mu'&=-\frac{b(b+2)}{2b+3} \\
  &=\frac{2a+3}{a+2}\cdot\frac{1}{a+2}\cdot\frac{-(a+2)}{a} \\
  &=-\frac{2a+3}{a(a+2)} \\
  &=1/\mu,
\end{align*}
and similarly $\lambda'=1/\lambda$.
\end{proof}

\begin{example} \label{ex:a_values}
Recall that for most choices of $\mu$
there are two values of $a$ (Remark~\ref{remark:special_value}).
The two corresponding $f_a$
may or may not belong to the same orbit.
For example, if we take $\mu=2$,
then we find that $a=-3\pm\sqrt3$.
But if $a=-3+\sqrt3$, then
$-a/(a+1)=-3-\sqrt3$,
and so it follows from the lemma that there is only one
orbit corresponding to $\mu=2$.

On the other hand, $\mu=5$ yields
$a=-6\pm\sqrt{21}$.
By the lemma, these two values of $a$ correspond to distinct orbits.
This justifies our earlier claim that
the function $s$ of Section~\ref{subs:s2}
is not sufficient to distinguish the closed orbits.
\end{example}

By analogy with the symmetrised cross-ratio
of \eqref{eq:symmXratio},
we define a function
$\pi \colon \complex\setminus\{0,-1,-3/2,-2,-3\}\to\complex$
by
\begin{align}
  \pi(a) &= \sigma_2( a, -\tfrac{2a+3}{a+2},  -(a+3),  -\tfrac{a}{a+1}, 
        -\tfrac{2a+3}{a+1},  -\tfrac{a+3}{a+2}) -\tfrac{117}{4} \notag \\
         &= \frac{a^2(3a+2)^2(a+3)^2}{(a+1)^2(a+2)^2}. \label{eq:pi}
\end{align}
The term $-\tfrac{117}{4}$ is again chosen to enable a nice factorisation,
and also ensures that $\pi(a)\to 0$ as $a\to 0$.

Define a map from $R_3^O$ to $\complex$,
also called $\pi$, by
\begin{equation} \label{eq:pi_on_R3}
\pi(f)=\pi(a)\text{ where } f \text{ is in the same orbit as } f_a.
\end{equation}
It follows from Lemma~\ref{lemma:a_transform}
and the use of the symmetric polynomial $\sigma_2$ in
\eqref{eq:pi}
that $\pi(f)$ is well defined.

\begin{lemma} \label{lemma:pi_holo}
The map $\pi$ is holomorphic on the open stratum.
\end{lemma}

\begin{proof}
Given $f\in R_3^O$
let $\lambda$ be the cross-ratio of the critical values of $f$
taken in some order,
and let $\mu$ be the cross-ratio of the critical points in the
corresponding order.
Then the value of $a$ is given by \eqref{eq:a_from_lambda}.
It follows that $\pi(f)$ is a holomorphic function
of $\lambda$ and~$\mu$.

A straightforward application of the argument principle
shows that, locally,
as $f$ varies holomorphically then so does each critical point,
and therefore so does each critical value.
Hence the dependence of $\mu$ and $\lambda$ on $f$ is holomorphic,
and so $\pi$ is holomorphic.
\end{proof}

\begin{remark}
The dependence of $f$ on $\mu$ and $\lambda$
can be described explicitly:
substituting the formulae of \eqref{eq:a_from_lambda}
into \eqref{eq:pi} yields
$$
\pi(\mu,\lambda)=\frac
      {p(\mu,\lambda)}
      {4 \lambda^2  (\lambda - 1)^2  \mu^4  (\mu - 1)^4}
$$
where
\begin{align*}
p(\mu,\lambda) =
& - \lambda^{2} \mu^{12} + 6 \lambda^{2} \mu^{11} + \lambda \mu^{12} + 160 
\lambda^{4} \mu^{8} + 54 \lambda^{3} \mu^{9} - 45 \lambda^{2} \mu^{10}\\
& - 4 \lambda \mu^{11} - 640 \lambda^{4} \mu^{7} - 563 \lambda^{3} \mu^{8}
+ 89 \lambda^{2} \mu^{9} + 34 \lambda \mu^{10} - \mu^{11}\\
& - 44 \lambda^{5} \mu^{5} + 
1044 \lambda^{4} \mu^{6} + 1676 \lambda^{3} \mu^{7} + 173 \lambda^{2} \mu^{8} - 
98 \lambda \mu^{9} + \mu^{10}\\
& + 110 \lambda^{5} \mu^{4} - 782 \lambda^{4} \mu^{5}
- 2340 \lambda^{3} \mu^{6} - 782 \lambda^{2} \mu^{7} + 110 \lambda \mu^{8}\\
& + \lambda^{6} \mu^{2} - 98 \lambda^{5} \mu^{3} + 173 \lambda^{4} 
\mu^{4} + 1676 \lambda^{3} \mu^{5} + 1044 \lambda^{2} \mu^{6} - 44 \lambda \mu^{7}\\
& - \lambda^{6} \mu + 34 \lambda^{5} \mu^{2} + 89 \lambda^{4} \mu^{3} - 
563 \lambda^{3} \mu^{4} - 640 \lambda^{2} \mu^{5} - 4 \lambda^{5} \mu\\
& - 45 
\lambda^{4} \mu^{2} + 54 \lambda^{3} \mu^{3} + 160 \lambda^{2} \mu^{4} + 
\lambda^{5} + 6 \lambda^{4} \mu - \lambda^{4}.
\end{align*}
\end{remark}

\begin{lemma} \label{lemma:same_pi}
Two elements $f_1$ and $f_2$ of $R_3^O$ are in the same orbit
if and only if $\pi(f_1)=\pi(f_2)$.
\end{lemma}

\begin{proof}
The forward implication is immediate from the definition.

Suppose $f_1$ is in the same orbit as $f_{a_1}$,
and $f_2$ in the same orbit as $f_{a_2}$.
Then a mechanical calculation gives
\begin{align*}
& \pi(f_2)-\pi(f_1) =  \\
& \tfrac{
    (a_1-a_2)(a_1+a_2+3)(a_1a_2+a_1+a_2)(a_1a_2+a_1+2a_2+3)
    (a_1a_2+2a_1+a_2+3)(a_1a_2+2a_1+2a_2+3)
    }
    {(a_1+1)^2(a_1+2)^2(a_2+1)^2(a_2+2)^2}.
\end{align*}

If the right hand side vanishes, then one of the factors of the numerator must vanish.
Each factor corresponds to one of the expressions
of Lemma~\ref{lemma:a_transform}.
For example,
$a_2=-(2a_1+3)/(a_1+1)$ is equivalent to
$a_1a_2+2a_1+a_2+3=0$.

Hence if $\pi(f_1)=\pi(f_2)$, then
Lemma~\ref{lemma:a_transform} implies that
$f_{a_1}$ and $f_{a_2}$ are in the same orbit,
and so $f_1$ and $f_2$ are in the same orbit.
\end{proof}

\begin{lemma} \label{lemma:image_of_pi}
The image of the open stratum under $\pi$
is $\cstar$.
\end{lemma}

\begin{proof}
We need to determine the values of $c\in\complex$ for which the equation
\begin{equation*}
\frac{a^2(3a+2)^2(a+3)^2}{(a+1)^2(a+2)^2} = c 
\end{equation*}
has a solution $a\in\complex\setminus\{0,-1,-3/2,-2,-3\}$.
We can multiply through to obtain
\begin{equation}
p(a,c)=a^2(3a+2)^2(a+3)^2 -c(a+1)^2(a+2)^2 = 0.\label{eq:image_of_pi}
\end{equation}
For fixed $c$, this is polynomial in $a$, and will
always have a solution in $\complex$.
For which $c$ does there exist a solution in $\complex\setminus\{0,-1,-3/2,-2,-3\}$?

If $a=-1$ or $-2$, then the second term of $p$ vanishes:
the value of $p$ is independent of $c$,
and is nonzero.
If $a=0$ or $-3/2$ or $-3$, then the first term of $p$ vanishes:
hence we find $c=0$.
Thus for nonzero $c$, the solutions $a$ of \eqref{eq:image_of_pi}
never lie in the set $\{0,-1,-3/2,-2,-3\}$,
and so $\pi$ maps $R_3^O$ onto~$\cstar$.
\end{proof}

\subsection{The quotient map} \label{subs:quotient}

In the previous section we defined a holomorphic map
$\pi\colon R_3^O\to\cstar$.
We extend $\pi$ to all of $R_3$
by defining $\pi(f)=0$ whenever $f$ is in the null fibre.

\begin{lemma} \label{lemma:pi_conts}
The map $\pi \colon R_3\to\complex$ is continuous.
\end{lemma}

\begin{proof}
We only need to prove continuity at the null fibre.
That is, we want to show that $\pi(f)\to0$
as $f$ approaches the null fibre.
First we give an intuitive picture of the situation;
a precise calculation follows.

We start by examining the limiting cases for the formulae
given in Lemma~\ref{lemma:standard_form}
as the parameter $a$ approaches a ``forbidden'' value or $\infty$.
Recall that the restrictions on $a$ arise
from requiring the critical points and critical values to be distinct.
As $a$ tends towards a forbidden value,
$\mu$, $\lambda$ and $f_a$ behave as in the following table.

\bigskip

\begin{tabular}{c|c|c|l}
$a$ & $\mu$ & $\lambda$ & $f$ tends to \\
\hline
0 & 0 & 0 & $x^3/(3x-2)$ \\
$-2$ & 0 & 0 & $x(2-x)$ \\
\hline
$-3$ & 1 & 1 & $x^2(x-3)/(1-3x)$ \\
$-1$ & 1 & 1 & $x^2$ \\
\hline
$-3/2$ & $\infty$ & $\infty$ & $x^2(3-2x)$ \\
$\infty$ & $\infty$ & $\infty$ & $x^2/(2x-1)$ \\
\end{tabular}

\bigskip

As $a$ approaches $-1$, $-2$ or $\infty$,
we see that the degree of $f_a$ drops,
so $f_a$ ``falls out of $R_3$'' as the critical points coalesce.
However, for the other cases,
$f_a$ approaches an element of the null fibre.

The idea of the proof, informally, is
that as we approach the null fibre in $R_3$,
the value of $a$ must approach $0$, $-3/2$ or $-3$.
These are exactly the values for which \eqref{eq:pi}
takes on the value~0.

Let us make this more precise.
Suppose $(f_n)_{n=1}^{\infty}$ is a sequence in $R_3^O$
tending to an element of the null fibre.
For each $n$, choose an ordering $\sigma_n$ of the critical points of $f_n$,
and Möbius transformations taking the critical points in order
to $\{0,1,\infty,\mu_n\}$
and the corresponding critical values to $\{0,1,\infty,\lambda_n\}$.
Furthermore, choose $\sigma_n$ so that $\mu_n$ is at least as close
to $0$ (in the spherical metric, say)
as it is to $1$ or $\infty$.

With respect to this choice of $\sigma_n$, we must have $\mu_n\to0$
as $n\to\infty$.

Using \eqref{eq:a_from_lambda},
if we let
$$ a_n = \frac{\mu_n^3+3\mu_n\lambda_n-4\lambda_n}{2\lambda_n(1-\mu)_n}, $$
then we have
$$\pi(f_n)= \frac{a_n^2(3a_n+2)^2(a_n+3)^2}{(a_n+1)^2(a_n+2)^2} =O(a_n^2).$$
Also
$$\mu_n=-\frac{a_n(a_n+2)}{2a_n+3}=O(a_n).$$
So as $n\to\infty$ and $\mu\to0$
we have $a_n\to0$ and therefore $\pi(f_n)\to0$.
\end{proof}

\begin{corollary} \label{cor:pi_holo}
The map $\pi \colon R_3\to\complex$ is holomorphic.
\end{corollary}

\begin{proof}
We already know that
$\pi$ is holomorphic outside the null fibre (Lemma~\ref{lemma:pi_holo}).
Recall (Remark~\ref{remark:null_codim})
that the null fibre has codimension~1.
Since $\pi$ is continuous at the null fibre,
it follows from Riemann's removable singularity theorem
that $\pi$ is holomorphic everywhere.
\end{proof}

\begin{remark} \label{remark:r3quotient}
Lemma~\ref{lemma:same_pi}
tells us that $\pi$ is constant on the orbits
and distinguishes the closed orbits,
and Lemma~\ref{lemma:image_of_pi} tells us that $\pi$ is surjective.
Informally, this means that
$\pi$ does the best possible job of distinguishing
the orbits---it is not possible for a holomorphic function
(or even a continuous function) to distinguish
the two orbits of the null fibre, since one is inside the closure
of the other---and so $\pi$ by itself forms a complete set of invariant functions
for the group action.
Therefore $\pi$ is the categorical quotient map,
proving Theorem~\ref{thm:r3quotient}.

This argument can be expressed more rigorously.
We know that the categorical quotient map $\pi' \colon R_3\to Y$
exists~\cite[page~70]{Snow-1982}.
Furthermore, $Y$ is a reduced complex space
and inherits from $R_3$
the properties of being connected, irreducible and normal~\cite[page~84]{Snow-1982}.
By the universal property of the quotient~\cite[Lemma~3.1]{Snow-1982},
there exists a unique holomorphic map $\alpha \colon Y\to\complex$
such that $\pi=\alpha\circ\pi'$.
From the above mentioned properties of $\pi$,
it follows that $\alpha$ is a bijection.
The difficulty is that $Y$ may have singularities,
so we cannot assume immediately that $\alpha^{-1}$ is holomorphic.

In our case, however,
$Y$ is irreducible and reduced, and therefore pure-dimensional.
Since there exists a holomorphic bijection $Y\to\complex$, the dimension must be~$1$.
But a $1$-dimensional normal space is necessarily smooth.
It follows that $\alpha$ is a biholomorphism,
and so $\pi \colon R_3\to\complex$ is also a categorical quotient map.
This completes the proof of Theorem~\ref{thm:r3quotient}.
\end{remark}

\subsection{Further structure: stabilisers and the exceptional orbit} \label{subs:stabilisers}

From the point of view of Oka theory,
we would like to know whether $\pi$
is an Oka map, in the sense defined in \cite[Definition~6.3]{Forstneric-Larusson-2011}.
A necessary condition is that $\pi$ should be a topological fibration.
In fact it is not.
This can be seen by studying the stabilisers of elements of $R_3$
and applying the results of \cite[Section~2.3]{Rainer-2009}.

Recall that if the cross-ratio of four points in some order
is $e^{\pm\pi i/3}$,
then the six cross-ratios listed in \eqref{eq:sixXratios}
take on only the two distinct values $e^{\pm\pi i/3}$.

\begin{definition}\label{def:exceptional_orbit}
The \emph{exceptional orbit}
is the set of elements of the open stratum of $R_3$
whose critical points, taken in some order,
have a cross-ratio of $e^{\pm\pi i/3}$.
\end{definition}

\begin{remark} \label{remark:exceptional_orbit}
The exceptional orbit is in fact a single orbit of the group action.
Straightforward calculations show that this orbit is
$\pi^{-1}(-27/4)$:
we can use \eqref{eq:a_mu_quadratic}
to find $a=-e^{\pm\pi i/3}-1$,
and for both choices of sign, \eqref{eq:pi} gives
the value $-27/4$.
Note also that $e^{\pm\pi i/3}$
are the special cross-ratio values referred to in Remark~\ref{remark:special_value}.
\end{remark}

\begin{theorem}\label{thm:stabilisers}
The stabilisers of elements of $R_3$ are of four types:
\begin{enumerate}[itemindent=-.1cm] 
  \item An element of the closed orbit in the null fibre has stabiliser of dimension~1.
        Hence this orbit is $5$-dimensional.
        The stabiliser is the nontrivial semidirect product of $\cstar$
        and $\integers_2$.
  \item An element of the non-closed orbit has finite stabiliser of size~2.
  \item An element of the exceptional orbit has finite stabiliser
          of size~12, isomorphic to the alternating group on four symbols.
  \item An element of the open stratum outside the exceptional orbit
          has finite stabiliser of size~4, isomorphic to the Klein 4-group.
          Furthermore, all such stabilisers are conjugate.
\end{enumerate}
\end{theorem}

\begin{remark}
By \cite[Corollary 5.5]{Snow-1982},
the restriction of the quotient map
to the open stratum minus the exceptional orbit
is a holomorphic fibre bundle.
The fibres are homogeneous spaces for $G$,
therefore Oka manifolds,
and so $\pi$ restricted to this domain is an Oka map.
This is the maximal open subset of $R_3$ over which
$\pi$ is a fibration.
\end{remark}

\begin{proof}[Proof of theorem]
Let $f\in R_3$ and $g=(\alpha,\beta)\in G$
such that $f^g=f$.
Then $\alpha^{-1}\circ f\circ\beta=f$.
In particular, $f$ and $\alpha^{-1}\circ f\circ\beta$
have the same critical points
and the same critical values.
It follows that $\alpha$ permutes the critical values of $f$
and $\beta$ permutes the critical points of $f$.
Furthermore, both permutations must preserve multiplicities.

\textit{Case (1):} We can choose $f=x^3$ as a representative of the
$5$-dimensional orbit.
The critical points are $0$ and $\infty$, each with multiplicity~2,
and the critical values are the same.
The only Möbius transformations
fixing the set $\{0,\infty\}$
are $x\mapsto cx$ and $x\mapsto c/x$ for $x\in\cstar$.
Choosing $\alpha$ and $\beta$ to be of this form,
and adding the restriction that $\alpha^{-1}\circ f\circ\beta=f$,
we obtain an explicit realisation of the stabiliser as
$$\{(c^3x,cx):c\in\cstar\}\cup \{(c^3/x,c/x):c\in\cstar\}.$$

\textit{Case (2):} Similarly, take $f=x^3+x^2$
as a representative of the non-closed orbit.
This has a double critical value of $\infty$ with preimage $\infty$,
and finite critical points and values $0\mapsto0$ and $-2/3\mapsto 4/27$.
(The finite critical points are simply the zeros of the derivative
$3x^2+2x$.)

Suppose $\alpha^{-1}\circ f \circ\beta=f$.
Then $\alpha$ and $\beta$ must both fix $\infty$.
This means that they are both of the form $x\mapsto ax+b$
for some $a\in\cstar$ and $b\in\complex$.
Also, $\alpha$ must either fix or interchange
the points $0$ and $4/27$.
Thus $\alpha$ is either the identity or the map $x\mapsto 4/27-x$.
Similarly, $\beta$ is either the identity or $x\mapsto -2/3-x$.

If we set $\alpha(x)=4/27-x$ and $\beta(x)=-2/3-x$,
noting that $\alpha^{-1}=\alpha$,
then we can calculate:
\begin{align*}
  \alpha\circ f &=4/27-f\neq f, \\
  f \circ\beta &= (-2/3-x)^2(-2/3-x+1) \\
               &= (4/9+4x/3+x^2)(1/3-x) \\
               &= 4/27-x^2-x^3 \neq f,\\
  \alpha\circ f\circ\beta &=4/27-(4/27-x^2-x^3) =f.
\end{align*}
and so the stabiliser is $\{(1,1),(\alpha,\beta)\}$
which has size~2 as stated above.

\textit{Cases (3) and (4), descriptions of the stabilisers:}
For orbits of $R_3^O$,
we can choose a representative $f$
as described in Lemma~\ref{lemma:standard_form},
with distinct critical points
$0$, $1$, $\infty$ and~$\mu$
and distinct critical values
$0$, $1$, $\infty$ and~$\lambda$.

Step 1: If $\alpha^{-1}\circ f\circ \beta=f$,
then $\alpha$ must permute the four critical points of $f$,
and $\beta$ must permute the four critical values;
furthermore, $\beta$ must induce the same permutation as $\alpha$.
Since a Möbius transformation is determined by the images of three points,
this greatly restricts the possibilities for $(\alpha,\beta)$.
To be specific, we can choose an ordering $(z_1,z_2,z_3,z_4)$
of $(0,1,\infty,t)$ (where $t$ can stand for $\lambda$ or $\mu$),
find the unique Möbius transformation $g$ sending $(0,1,\infty)$
to $(z_1,z_2,z_3)$, and check whether $g(t)=z_4$.
The 24 possibilities are listed in Appendix~\ref{appendix:table}.

For generic values of $t$,
there are only four permutations,
given by the rows of the table
with ``any'' in the fourth column.
For each permutation
we can calculate $\alpha^{-1}\circ f\circ \beta=f$ explicitly
and verify that $(\alpha,\beta)$ does indeed stabilise $f$.
The nontrivial permutations are all pairs of transpositions,
giving the Klein 4-group.
This proves case~(4).

We obtain additional permutations only when $\lambda$
and $\mu$ are both special cross-ratio values,
i.e.\ one of $-1$, $\tfrac{1}{2}$, 2 or~$e^{\pm\pi i/3}$.

Step 2: If either $\mu$ or $\lambda$ is not one of the above special values,
then the only possible elements of the stabiliser are
those identified in Step~1 above.
So we need to check whether $\mu$ and $\lambda$ can be simultaneously special.
This is straightforward:
for each value of $\mu$ we solve \eqref{eq:a_mu_quadratic} above
to find the corresponding values of $a$, and then calculate $\lambda$.
The result is that if $\mu$ is one of $-1$, $\tfrac{1}{2}$ or $2$,
then $\lambda$ is real and irrational, therefore not special,
but if $\mu$ is $e^{\pm\pi i/3}$, then $\lambda=\bar\mu$ is special.
In this case there are an additional eight candidate elements in the stabiliser.

The calculations for the special values of $\lambda$ and $\mu$
are summarised in the following table.

\bigskip
{\tiny
\begin{tabular}{c|c|c|c}
$\mu$ & equation & $a$ & $\lambda=\mu^3/(a+2)^2$ \\
\hline
$-1$ & $a^2-3=0 $& $\pm\sqrt 3$ & real and irrational \\
\hline
$1/2$ & $a^2+3a+3/2 =0$ & $(-3\pm\sqrt 3)/2$ & real and irrational \\
\hline
$2$ & $a^2+6a+6=0$ & $-3\pm\sqrt 3$ & real and irrational \\
\hline
$e^{\pi i/3}$ & $a^2+2(e^{\pi i/3}+1)+3e^{\pi i/3}=0$ & $-e^{\pi i/3}-1$ & $1-e^{\pi i/3}=e^{-\pi i/3}$ \\
\hline
$e^{-\pi i/3}$ & $a^2+2(e^{-\pi i/3}+1)+3e^{-\pi i/3}=0$ & $-e^{-\pi i/3}-1$ & $1-e^{-\pi i/3}=e^{\pi i/3}$
\end{tabular}
} 
\bigskip

Step 3: For each of the candidate elements identified above,
calculate $\alpha^{-1}\circ f \circ \beta$ and verify that it equals $f$.
(In fact, knowing that the stabiliser is a group,
we only need to verify this for one element outside the generic stabiliser.
It is easiest to work with the permutation $(0 1 \infty)$,
for which $\alpha(x)=\beta(x)=1/(1-x)$ and $\alpha^{-1}(x)=(x-1)/x$.)
Hence the stabiliser of the exceptional orbit
has size~$12$.
From the table in Appendix~\ref{appendix:table}
we see that all elements of this stabiliser
induce even permutations on the set
$\{0,1,\infty,e^{\pi i/3}\}$,
and so the stabiliser is
isomorphic to the alternating group on four symbols.

\textit{Case (4), conjugacy of stabilisers:}
Every finite subgroup of the Möbius group
is conjugate to a subgroup of the group
$\PSU_2(\complex)$, which can be viewed as the group of rigid motions
of the Riemann sphere with respect to the usual embedding into $\reals^3$.
See for example \cite{Lyndon-Ullman-1967} or \cite[Section~2.13]{Jones-Singerman-1987}.

Two finite subgroups
of the Möbius group are conjugate if and only if they are isomorphic as abstract groups
(\cite[remarks after Corollary~2.13.7]{Jones-Singerman-1987}).
However, a slightly stronger result is needed for our purposes.

For a Klein 4-subgroup of $\PSU_2(\complex)$,
viewed as a group of rigid motions of the sphere,
each non-identity element is a rotation by an angle of $\pi$ about some axis.
It is clear that two rotations commute if and only if their axes are orthogonal.
Thus Klein 4-groups correspond to sets of three mutually orthogonal axes.
For any two such sets of axes,
there is an orientation-preserving rigid motion of the sphere taking one to the other.
This gives a group element conjugating one Klein 4-group to the other.
Therefore any two such subgroups are conjugate.

We can go a little further.  For a set of three mutually orthogonal axes,
and for any permutation of those axes,
there exists a rotation realising that permutation.
Conjugating by this rotation will yield an automorphism of the corresponding
Klein 4-group which permutes the non-identity elements in the same way.

Hence we can conclude that
given Klein 4-subgroups $\{1,\alpha_1,\alpha_2,\alpha_3\}$ and $\{1,\beta_1,\beta_2,\beta_3\}$
of the Möbius group,
there exists a group element $g$ with $g\alpha_jg^{-1}=\beta_j$ for $j=1,2,3$.
This is the stronger result referred to above.

Now we apply this to stabilisers in
$\mbsgp\times\mbsgp$.
If $f$ is in the open stratum but not in the exceptional orbit,
then the stabiliser is of the form
$\{(1,1),(\alpha_1,\beta_1),(\alpha_2,\beta_2),(\alpha_3,\beta_3)\}$,
where each $\alpha_j$ permutes the critical values of $f$
via a pair of disjoint transpositions,
and each $\beta_j$ carries out the same permutation on the corresponding
critical points.
Given another stabiliser of the form
$\{(1,1),(\alpha'_1,\beta'_1),(\alpha'_2,\beta'_2),(\alpha'_3,\beta'_3)\}$,
we seek $(g,h)$ such that
$g$ conjugates $\{1,\alpha_1,\alpha_2,\alpha_3\}$
to $\{1,\alpha'_1,\alpha'_2,\alpha'_3\}$ in some order, and
$h$ conjugates $\{1,\beta_1,\beta_2,\beta_3\}$
to $\{1,\beta'_1,\beta'_2,\beta'_3\}$ in the same order.
The fact that any two isomorphic finite subgroups of $\mbsgp$ are conjugate
tells us that a suitable $g$ exists.
Then the existence of $h$ is guaranteed by the stronger result
that we can conjugate the elements of one Klein 4-subgroup to another
in any desired order.
\end{proof}

\section{Degree 3: dominability and $\complex$-connectedness} \label{section:dominating}

\subsection{Composition of dominability} \label{subs:composition}

When exploring the Oka property or related flexibility properties of a manifold $X$,
a logical first step is to investigate holomorphic maps $\cn\to X$,
and in particular to look for dominating maps (Definition~\ref{def:dominable}).

In the case of $R_d$, the principal difficulty in constructing
explicit dominating maps is cancellation.
The easiest way to write down a map $\cn\to R_d$
is in the form $p(t)/q(t)$ where $p$ and $q$
are families of polynomials parametrised by $t\in\cn$.
However, it is necessary to ensure that $p$ and $q$
do not have common factors as the parameter $t$ varies.
We can achieve this by 
embedding $R_d$ in a larger space and then
applying the following result.

\begin{proposition}[Composition of dominability] \label{prop:composition}
Let $X$ be an open subset of a complex manifold $Z$,
and $p\in X$.
Suppose $\phi \colon \cn\to Z$ dominates $Z$ at $p$.
If $\phi^{-1}(X)$ is dominable at 0, then $X$ is dominable at $p$.
\end{proposition}

\begin{proof}
If $\psi \colon \complex^m\to \phi^{-1}(X)$ dominates $\phi^{-1}(X)$ at 0,
then $\phi\circ\psi$ dominates $X$ at $p$.
\end{proof}

We construct a map $\complex^8\to R_3$ as follows.

We can view a point of $\pseven$ as a formal rational function:
$$(a_0 \colon\! \cdots \colon\! a_7) \longleftrightarrow 
\frac{a_0x^3+a_1x^2+a_2x+a_3}{a_4x^3+a_5x^2+a_6x+a_7}.$$
We also have the group $G=\mbsgp\times\mbsgp$
acting on $R_3$ by pre- and post-composition.
This extends to an action of $G$ on $\pseven$:
we can compose a formal rational function with a Möbius transformation
to get a well-defined result.

Recall that we can embed $R_3$ into $\pseven$
by using the coefficients of a rational function
as the homogeneous coordinates of a point of $\pseven$.
This embedding is $G$-equivariant;
in the following discussion we will identify $R_3$ with its image in $\pseven$.

Now choose $f\in R_3$
and suppose we have a map $\eta \colon \complex^2\to\pseven$
sending 0 to $f$.
Let $\exp \colon \complex^6\to G$ be the exponential map.
This map dominates $G$ at the identity;
for this particular group, it is also surjective~\cite[page~47]{Gorbatsevich-et-al-1997}.
Define $\phi \colon \complex^8\to\pseven$ by
\begin{equation}
  \phi(s,t)=\eta(s)^{\exp(t)},\qquad s\in\complex^2,\> t\in\complex^6.
  \label{eq:phi}
\end{equation}
Our strategy is to choose $\eta$ so that $\phi$ dominates $\pseven$ at $f$,
and find some $\psi \colon \complex^m\to\phi^{-1}(R_3)$ which is dominating at 0.
The proposition then tells us that $\phi\circ\psi$ dominates $R_3$ at $f$.
This proves Theorem~\ref{thm:r3dominable}.

Furthermore, $\psi$ can be chosen so that
$\phi\circ\psi$ is surjective (Corollary~\ref{cor:comp_surjective}).
Since $\complex^8$ is $\complex$-connected (Definition~\ref{def:connected}),
it follows that $R_3$ is $\complex$-connected,
proving Theorem~\ref{thm:r3connected}.

To construct a suitable $\eta$,
first we define $\eta_0 \colon \complex^2\to\pseven$ by
\begin{equation}
  \eta_0(a,b)=\frac{x^3-ax}{-bx^2+1}
  =(1 \colon\! 0 \colon\!\! -a \colon\! 0 \colon\!
    0 \colon\!\! -b \colon\! 0 \colon\! 1). \label{eq:eta}
\end{equation}
We will see in Proposition~\ref{prop:all_orbits}
that the image of $\eta_0$ intersects every orbit
of $G$ on $R_3$.
Thus given $f\in R_3$ there exist
$a_0,b_0\in\complex$ and $g\in G$
such that $g$ takes $\eta_0(a_0,b_0)$ to $f$.
Define $\eta$ by
\begin{equation*}
  \eta(a,b)=\eta_0(a+a_0,b+b_0)^g.
\end{equation*}

\begin{remark}
The form of $\eta$ is not uniquely determined by the choice of $f$.
For the purpose of proving strong dominability,
this does not matter:
all we need is that given $f$
there exists at least one suitable $\eta$.
If we could in fact find a canonical $\eta$ for each $f$,
in such a way that the map $f\mapsto\eta$ were holomorphic,
then we could join the resulting dominating maps
to make a spray (\cite[Definition~5.1]{Forstneric-Larusson-2011}).
This would imply that $R_3$ is Oka.
\end{remark}

The numerator of $\eta_0$ has roots $0$ and $\pm\sqrt{a}$,
and the denominator has roots $\pm1/\sqrt{b}$.
Therefore $\eta_0(a,b)$ fails to be in $R_3$ exactly when $ab=1$.
Similarly, given $a_0$, $b_0$ and $g$,
the set of $(a,b)$ such that $\eta(a,b)\not\in R_3$
is a translate of $\complex^2\setminus\{(a,b):ab=1\}$.
Therefore
$$\phi^{-1}(R_3)\cong(\complex^2\setminus\{(a,b):ab=1\})\times\complex^6.$$
Now $\complex^2\setminus\{ab=1\}$ is Oka:
this is a consequence of \cite[Proposition~4.10]{Hanysz-2012},
or see Appendix~\ref{appendix:conic} for an elementary proof.
In particular, $\complex^2\setminus\{ab=1\}$,
and hence $\phi^{-1}(R_3)$,
is dominable at $0$.
Thus $\phi^{-1}(R_3)$ is dominable at $0$.
We will show in Section~\ref{subs:transverse}
that $\phi$ dominates $\pseven$ at $f$.
Therefore $\phi\circ\psi$ dominates $R_3$ at~$f$.

\subsection{Proof of surjectivity} \label{subs:surjective}

The goal of this section
is to show that
the map $\psi \colon \complex^8\to\phi^{-1}(R_3)$
can be chosen so that $\phi\circ\psi$ is surjective.
First we prove that the image of the map $\eta_0$
defined by \eqref{eq:eta}
intersects every orbit,
and therefore $\phi$ is surjective.
Then we will describe the choice of $\psi$.

For the first part, we exploit the fact
that the critical values of $\eta_0(a,b)$
have a certain kind of symmetry.

\begin{definition} \label{def:balanced}
Let $z_1,\ldots,z_4\in\complex$.
We say that $(z_1,\ldots,z_4)$ is \emph{balanced}
if $z_1+z_2=z_3+z_4=0$.
\end{definition}

\begin{lemma} \label{lemma:balanced}
  Let $z_1,\ldots,z_4$ be distinct points of $\complex$.
  There exists a Möbius transformation $\alpha$
  such that $(\alpha(z_1),\ldots,\alpha(z_4))$
  is balanced.
\end{lemma}

\begin{proof}
By transitivity of the Möbius group,
we can assume that $z_3=1$ and $z_4=-1$.
We will find a Möbius transformation
$\alpha$ fixing $1$ and $-1$,
and such that $\alpha(z_1)+\alpha(z_2)=0$.
Suppose
$$\alpha(x)=\frac{ax+b}{cx+d}.$$
Then $\alpha(1)=1$ tells us that $a+b=c+d$,
and $\alpha(-1)=-1$ implies $a-b=c-d$.
Hence $a=d$ and $b=c$, so $\alpha$ is of the form
$$\alpha(x)=\frac{ax+b}{bx+a}$$
for some $a,b\in\complex$.
For $\alpha$ to be invertible
we also need $a\neq\pm b$.

We wish to find $a$ and $b$ such that
$$\frac{az_1+b}{bz_1+a}+\frac{az_2+b}{bz_2+a}=0.$$
This gives
$$(z_1+z_2)a^2+2(1+z_1z_2)ab+(z_1+z_2)b^2=0,$$
or, setting $A=a/b$ or $A=b/a$
(by symmetry, both are possible),
$$(z_1+z_2)A^2+2(1+z_1z_2)A+(z_1+z_2)=0.$$
This always has a solution for $A$.
The condition $a\neq\pm b$ means that we require $A\neq\pm 1$.
But if $A=1$, then the left hand side of the equation is
$$2(z_1+z_2+1+z_1z_2)=2(z_1+1)(z_2+1),$$
which is always nonzero when $z_1,z_2,\pm1$ are distinct.
Similarly, $A=-1$ will also give a nonzero left hand side.
Hence it is always possible to find $a$ and $b$
satisfying the required conditions.
\end{proof}

The proof of the following elementary result
is left as an exercise for the reader.

\begin{lemma} \label{lemma:odd}
Let $U\subset\complex$ be a connected open set containing 0
and such that $-x\in U$ for every $x\in U$.
Let $f \colon U\to\complex$ be a holomorphic function
such that $f'$ is even and $f(0)=0$.
Then $f$ is odd.
\end{lemma}

\begin{lemma} \label{lemma:odd_over_even}
Let $f=p/q\in R_d$
where $p$ and $q$ are polynomials with no common factors.
Suppose $f$ is an odd function.
Then either $p$ is odd and $q$ is even,
or $p$ is even and $q$ is odd.
\end{lemma}

\begin{proof}
Since $p$ and $q$ have no common factors,
the zeros of $p$ are precisely the zeros of $f$ in $\complex$.
Since $f$ is odd, the zeros are distributed
symmetrically about $0$,
and so $p$ is either odd or even
(depending on whether $f(0)=0$ or $\infty$).
Similarly, $q$ is either even or odd.
\end{proof}

\begin{corollary} \label{cor:odd_over_even}
If $f\in R_3$ is odd and $f(0)=0$,
then $f$ can be written in the form
$$f(x)=\frac{Ax^3+Bx}{Cx^2+1}$$
for some $A,B,C\in\complex$.
\end{corollary}

\begin{proposition} \label{prop:all_orbits}
The image of the map $\eta_0$
of \eqref{eq:eta}
intersects every orbit in $R_3$.
\end{proposition}

\begin{proof}
First, note that $\eta_0(0,0)=x^3$ is in the small orbit
and $\eta_0(1,0)=x^3-x$ is in the non-closed orbit.
Therefore we only need to consider orbits outside the null fibre.

Given $f\in R_3$ outside the null fibre,
Lemma~\ref{lemma:balanced} ensures that there is a Möbius transformation $\alpha$
such that the critical points of $f\circ\alpha$ are balanced.
(In fact $\alpha$ is the inverse of a transformation
taking the critical points to a balanced quadruple.
If one of the critical points of $f$ is $\infty$,
then before applying the lemma we precompose $f$
with a suitable transformation so that the resulting critical points
are all finite.)

Now the finite critical points of $f$ are the zeros of $f'$, that is,
the roots of the numerator of $f'$.
For $f=p/q$ we have $f'=(p'q-pq')/p^2$.
After balancing the critical points,
the numerator of $(f\circ\alpha)'$ will be of the form
$C(x^2-A^2)(x^2-B^2)$ for some $A,B,C\in\cstar$.
The denominator is a perfect square.
Therefore $(f\circ\alpha)'$ is an even function.

Let $\beta$ be a Möbius transformation taking $f(\alpha(0))$ to 0,
so that $\beta\circ f\circ\alpha$ has the same critical points as $f\circ\alpha$.
Let
$U=\complex\setminus\{x\in\complex:
x\text{ or }{-x}\text{ is a pole of }\beta\circ f\circ\alpha\}.$
Then $\beta\circ f\circ\alpha|_U$ satisfies the conditions
of Lemma~\ref{lemma:odd}, and is therefore an odd function.
By continuity, it follows that any poles of $\beta\circ f\circ\alpha$
must be symmetrically distributed about 0,
and $\beta\circ f\circ\alpha$ is odd.
By Corollary~\ref{cor:odd_over_even} we have
$$(\beta\circ f\circ\alpha)(x)=\frac{Ax^3+Bx}{Cx^2+1}$$
for some $A,B,C\in\complex$.
Also, $f\in R_3$ implies $A\neq0$.
Therefore we can postcompose with the Möbius transformation
$x\mapsto x/A$
to see that $f$ is in the same orbit as $\eta_0(-B/A,-C)$.
\end{proof}

\begin{corollary} \label{cor:surjective}
The image of the map $\phi$ of \eqref{eq:phi} is exactly $R_3$.
\end{corollary}

Now we describe the map $\psi \colon \complex^8\to\phi^{-1}(R_3)$.
Since the exponential map $\exp \colon \complex^6\to G$ is surjective,
we simply need to find a surjective holomorphic map
$\chi \colon \complex^2\to\complex^2\setminus\{ab=1\}$
which is dominating at~$0$,
and then we can take $\psi=\chi\times\exp$.

Following Buzzard and Lu~\cite[page~645]{Buzzard-Lu-2000},
define a map $\omega \colon \complex^2\to\complex$ by
\begin{align*}
    \omega(x,y)&=\left\{
    \begin{array}{cl}
        \dfrac{e^{xy}-1}{x} & \textup{if }x\neq0 \label{eq:BLmap} \\
        y & \textup{if }x=0
    \end{array}
    \right. \\
    &= y+\frac{xy^2}{2}+\frac{x^2y^3}{3!}+\cdots.
\end{align*}
From the first form of the definition,
we can see that for fixed $x$,
the image of $y\mapsto\omega(x,y)$ is $\complex\setminus\{-1/x\}$.
From the series expression we can see that $\omega$ is holomorphic,
and we can calculate derivatives
\begin{align*}
    \frac{\partial\omega}{\partial x}\bigg\vert_{x=0} & = y^2/2, \\
    \frac{\partial\omega}{\partial y} &= e^{xy}.
\end{align*}

\begin{proposition} \label{prop:chi_surjective}
The map $\chi \colon \complex^2\to\complex^2\setminus\{ab=1\}$ defined by
$$\chi(x,y)=\left(x,\frac{1-e^{xy}}{x}\right)=(x,-\omega(x,y))$$
is surjective and dominating at~$0$.
\end{proposition}

\begin{proof}
Recall that for fixed $a$, the image of
$y\mapsto\omega(a,y)$ is $\complex\setminus\{-1/a\}$.
Thus if $ab\neq1$, then there exists $y$ such that
$\omega(a,y)\neq -b$,
and then $\chi(a,y)=(a,b)$.
Hence $\chi$ is surjective.

To prove dominability, we need to verify that the
vectors $\partial\chi/\partial x$ and $\partial\chi/\partial y$
evaluated at $(x,y)=(0,0)$ span $\complex^2$.
But we have
$\partial\chi/\partial x\vert_{x=0}= (1,y^2/2)$
and $\partial\chi/\partial y=(0,e^{xy})$.
Thus the derivatives evaluated at $(0,0)$
are $(1,0)$ and $(0,1)$.
\end{proof}

\begin{corollary} \label{cor:comp_surjective}
With $\chi$ as above, $\psi=\chi\times\exp$
and $\phi$ as defined in \eqref{eq:phi},
the composition $\phi\circ\psi$ is surjective.
\end{corollary}

\subsection{Proof of transversality} \label{subs:transverse}

We wish to show that the map $\phi$
of \eqref{eq:phi} is dominating.
Recall that $\phi \colon \complex^8\to\pseven$ is built
from maps $\eta \colon \complex^2\to\pseven$
and $\exp \colon \complex^6\to G$.
For convenience, we repeat the definitions here.
\begin{align*}
  \eta_0(a,b)&=\frac{x^3-ax}{-bx^2+1}
  =(1 \colon\! 0 \colon\!\! -a \colon\! 0 \colon\! 0 \colon\!\! -b \colon\! 0 \colon\! 1), \\
  \eta(a,b)&=\eta_0(a+a_0,b+b_0)^g\text{ for some constants }
             a_0,b_0\in\complex\text{ and }g\in G, \\
  \phi(s,t)&=\eta(s)^{\exp(t)},\qquad s\in\complex^2,\> t\in\complex^6. \\
\end{align*}

Since $\exp$ is dominating at the identity,
it follows that the image of $d\phi_0$ contains the tangent space
to the fibre of $\eta(0,0)$.
So to prove that $\phi$ is dominating,
it is sufficient to show that the image of $d\eta$ is transverse to the
tangent space of the fibre.
More precisely, we will show that the images of $d\phi_0$ and $d\eta_{(0,0)}$
together span the tangent space $T\pseven_f$,
where $f=\phi(0,0)$.

\begin{proposition} \label{prop:transverse}
Let $a,b\in\complex$, $ab\neq1$, and $f=\eta_0(a,b)$.
Suppose $f\in R_3$.
Then the tangent space $T\pseven_f$
is spanned by the tangent space to the fibre of $f$
together with the image of $d\eta_0{(a,b)}$.
\end{proposition}

\begin{proof}
For convenience, we will work in affine coordinates:
rational functions of the form
$$\frac{a_0x^3+a_1x^2+a_2x+a_3}{b_0x^3+b_1x^2+b_2x+1}$$
will be written as
$(a_0,a_1,a_2,a_3;b_0,b_1,b_2)$.
In this notation, we have
$$\eta_0(a,b)=(1,0,-a,0;0,-b,0).$$
We can identify $T\pseven_f$ with $\complex^7$,
and $d\eta_{(a,b)}$ is the subspace
\begin{equation}
\{(0,0,u,0;0,v,0):u,v\in\complex\}. \label{eq:transverse}
\end{equation}
The main part of the proof consists in finding a set of vectors
spanning the tangent space to the fibre.
Such a set can be realised as derivatives of
infinitesimal generators of the group.

The Lie algebra $\Liesl_2(\complex)$ is generated by the three matrices
$$
  A=\begin{pmatrix} 1 & 0 \\ 0 & -1 \end{pmatrix} \qquad
  B=\begin{pmatrix} 0 & 1 \\ 0 & 0 \end{pmatrix} \qquad
  C=\begin{pmatrix} 0 & 0 \\ 1 & 0 \end{pmatrix}.
$$
Infinitesimal generators of $\mbsgp$
are given by their exponentials:
$$
  e^{At}=\begin{pmatrix} e^t & 0 \\ 0 & e^-t \end{pmatrix} \qquad
  e^{Bt}=\begin{pmatrix} 1 & t \\ 0 & 1 \end{pmatrix} \qquad
  e^{Ct}=\begin{pmatrix} 1 & 0 \\ t & 1 \end{pmatrix}
$$
or, interpreted as Möbius transformations:
$$e^{At} \colon x\mapsto e^{2t}x \qquad
  e^{Bt} \colon x\mapsto x+t \qquad
  e^{Ct} \colon x\mapsto \frac{x}{tx+1}.
$$
As infinitesimal generators
of the group $G=\mbsgp\times\mbsgp$,
we can use ordered pairs
$(g^{-1},\text{id})$
and $(\text{id},g)$,
where $g$ is one of $e^{At}$, $e^{Bt}$ or $e^{Ct}$.
Thus we simply need to calculate the six vectors
$$\frac{d}{dt}(\eta_0(a,b)\circ e^{At})|_{t=0}, \qquad
  \frac{d}{dt}(e^{At}\circ\eta_0(a,b))|_{t=0},$$
and the corresponding vectors for $e^{Bt}$ and $e^{Ct}$.

The first of those six vectors is computed as follows:
$$\eta_0(a,b)\circ e^{At} \colon x\mapsto
  \frac{e^{6t}x^3-ae^{2t}x}{-be^{4t}x^2+1}
  =(e^{6t},0,-ae^{2t},0;0.-be^{4t},0),
$$
and so the derivative with respect to $t$,
evaluated at $t=0$, is
$$(6,0,-2,0;0,-4b,0).$$
The remaining cases are handled similarly.
The end result of the calculation
is that the tangent space to the fibre
is spanned by the rows of the following matrix:
$$\begin{pmatrix}
6   &  0   &  -2a &  0   &  0   &  -4b &  0   \\ 
0   &  3   &  0   &  -a  &  0   &  0   &  -2b \\ 
0   &  -2a &  0   &  0   &  -b  &  0   &  3   \\ 
2   &  0   &  -2a &  0   &  0   &  0   &  0   \\ 
0   &  -b  &  0   &  1   &  0   &  0   &  0   \\ 
0   &  0   &  0   &  0   &  1   &  0   &  a
\end{pmatrix}
$$
and it is straightforward to verify that
the rows together with the vectors of \eqref{eq:transverse}
span $\complex^7$.
\end{proof}

Since the above calculation does not depend on the choice of $a$ and $b$,
we have a dominating map for every $f\in R_3$,
proving Theorem~\ref{thm:r3dominable}.

\appendix

\section{Table of Möbius transformations} \label{appendix:table}

Let $t\in\pone\setminus\{0,1,\infty\}$.
Then for each sequence of three distinct elements
$z_1,z_2,z_3$
of $\{0,1,\infty,t\}$
there is a unique Möbius transformation $g$
sending $0$, $1$ and~$\infty$
to $z_1$, $z_2$ and~$z_3$ respectively.
Table~1 lists the form of $g$ for all 24 possible choices
of $(z_1,z_2,z_3)$.
The fourth column
contains the values $t_0$
such that $g(t_0)\in\{0,1,\infty,t_0\}$,
and the last column
gives the permutation induced by $g$ on the set
$\{0,1,\infty,t\}$ when $t=t_0$.
This table is used in the proof of Theorem~\ref{thm:stabilisers}.

\bigskip

\begin{table}
\caption{The 24 Möbius transformations of Appendix~\ref{appendix:table}}
\begin{tabular}{c|c|c|c|c}
$g(0,1,\infty)$  & $g(x)$ & $g(t)$ & $t_0$ & permutation \\
\hline\hline
$(0 , 1 , \infty)  $ & $x$ & $t$ & any & id \\ \hline
$(0 , 1 , t)       $ & $tx/(x+t-1)$ & $t^2/(2t-1)$ & 1/2 & $(t \infty)$ \\ \hline
$(0 , \infty , 1)  $ & $x/(x-1)$ & $t/(t-1)$ & 2 & $(1 \infty)$ \\ \hline
$(0 , \infty , t)$ & $tx/(x-1)$ & $t^2/(t-1)$ & $e^{\pm\pi i/3}$ & $(1 \infty t)$ \\ \hline
$(0 , t , 1)       $ & $tx/(tx+1-t)$ & $t^2/(t^2-t+1)$ & $e^{\pm\pi i/3}$ & $(1 t \infty)$ \\ \hline
$(0 , t , \infty)$ & $tx$ & $t^2$ & $-1$ & $(1 t)$ \\ \hline
$(1 , 0 , \infty)  $ & $1-x$ & $1-t$ & 1/2 & $(0 1)$ \\ \hline
$(1 , 0 , t)       $ & $(tx-t)/(x-t)$ & $\infty$ & any & $(0 1)(\infty t)$ \\ \hline
$(1 , \infty , 0)  $ & $1/(1-x)$ & $1/(1-t)$ & $e^{\pm\pi i/3}$ & $(0 1 \infty)$ \\ \hline
$(1 , \infty , t)$ & $(tx-1)/(x-1)$ & $t+1$ & $-1$ & $(0 1 \infty t)$ \\ \hline
$(1 , t , 0)       $ & $t/((1-t)x+t)$ & $1/(2-t)$ & 2 & $(0 1 t \infty)$ \\ \hline
$(1 , t , \infty)$ & $(t-1)x+1$ & $t^2-t+1$ & $e^{\pm\pi i/3}$ & $(0 1 t)$ \\ \hline
$(\infty , 0 , 1)  $ & $(x-1)/x$ & $(t-1)/t$ & $e^{\pm\pi i/3}$ & $(0 \infty 1)$ \\ \hline
$(\infty , 0 , t)$ & $t(x-1)/x$ & $t-1$ & 2 & $(0 \infty t 1)$ \\ \hline
$(\infty , 1 , 0)  $ & $1/x$ & $1/t$ & $-1$ & $(0 \infty)$ \\ \hline
$(\infty , 1 , t)$ & $(tx+1-t)/x$ & $(t^2-t+1)/t$ & $e^{\pm\pi i/3}$ & $(0 \infty t)$ \\ \hline
$(\infty , t , 0)$ & $t/x$ & $1$ & any & $(0 \infty)(1 t)$ \\ \hline
$(\infty , t , 1)$ & $(x+t-1)/x$ & $(2t-1)/t$ & 1/2 & $(0 \infty 1 t)$ \\ \hline
$($t$ , 0 , 1)       $ & $t(x-1)/(tx-1)$ & $t/(t+1)$ & $-1$ & $(0 t \infty 1)$ \\ \hline
$($t$ , 0 , \infty)$ & $t(1-x)$ & $t(1-t)$ & $e^{\pm\pi i/3}$ & $(0 t 1)$ \\ \hline
$($t$ , 1 , 0)       $ & $t/((t-1)x+1)$ & $t/(t^2-t+1)$ & $e^{\pm\pi i/3}$ & $(0 t \infty)$ \\ \hline
$($t$ , 1 , \infty)$ & $(1-t)x+t$ & $2t-t^2$ & 2 & $(0 t)$ \\ \hline
$($t$ , \infty , 0)$ & $t/(1-x)$ & $t/(1-t)$ & 1/2 & $(0 t 1 \infty)$ \\ \hline
$($t$ , \infty , 1)$ & $(x-t)/(x-1)$ & $0$ & any & $(0 t)(1 \infty)$ \\ \hline
\end{tabular}
\end{table}

\section{The Oka property for the complement of an affine plane conic}
\label{appendix:conic}

It is well known that the complement of a smooth quadric surface
in $\projspace{n}$ is a homogeneous manifold,
and is therefore Oka.
In affine space, the situation is more complicated.
Here we discuss only the case of smooth quadratic curves
in $\complex^2$,
as mentioned at the end of Section~\ref{subs:composition}.
There is scope for further investigation in higher dimensions.

In $\complex^2$ there are two isomorphism classes
of smooth conics, depending on the behaviour at infinity.
One class is represented by the curve $y=x^2$.
For the complement of this curve,
the projection onto the $x$-axis has fibre $\cstar$.
It is readily seen that this projection
is a trivial fibre bundle.
Thus the complement is
a fibre bundle with Oka base and Oka fibre,
and is therefore an Oka manifold.

The more interesting case is the complement of $xy=1$.
Here the projection map onto a coordinate axis
is not a fibration, because both $\cstar$ and $\complex$ occur as fibres.
Therefore it is rather more difficult to prove the Oka property.
We can give a reasonably short proof by constructing a suitable
covering space which ``untwists'' the fibres,
and using the fact that a manifold is Oka if and only if it has
an Oka covering space.
We construct sprays (in the sense of Gromov) on the covering space,
proving that it is subelliptic, which implies the Oka property.
Refer to~\cite[Section~5]{Forstneric-Larusson-2011} for the relevant definitions.

Write $X=\complex^2\setminus\{xy=1\}$ and
$Z=\{(x,y,z)\in\complex^3:e^z=xy-1\}$.
Define $\pi \colon Z\to X$ by
$\pi(x,y,z) = (x,y)$.
Then $\pi$ is a covering map, so $Z$ is Oka if and only if $X$ is Oka.

Define maps $s_1,s_2,s_3 \colon Z\times\complex\to Z$ by
\begin{align*}
    s_1(x,y,z,t)&=(e^t x, e^{-t} y, z),\\
    s_2(x,y,z,t)&=\left(\frac{1+e^{z+ty}}{y}, y, z+ty\right),\\
    s_3(x,y,z,t)&=\left(x,\frac{1+e^{z+tx}}{x}, z+tx\right).
\end{align*}
Note that $s_2$ and $s_3$ are holomorphic everywhere;
the behaviour near $x=0$ or $y=0$
is similar to that of the map $\omega$ discussed
immediately before Proposition~\ref{prop:chi_surjective}.

It is straightforward to check from the definitions
that the $s_j$ are sprays
and that they dominate at every point of $Z$.
Hence $Z$ is subelliptic,
and so $Z$, and therefore $X$, is Oka.

\section{Sage code for computations} \label{appendix:sage}

Although the calculations in this paper are easy enough
to verify by hand once the answer is known,
the computer has been very useful as an exploratory tool
in trying out various possibilities.
The availability of computer algebra systems means that
the research proceeded more quickly than would otherwise have been possible.

This appendix reproduces some of the code
that was used to find the results of
Sections~\ref{subs:s2}, \ref{subs:standard},
\ref{subs:cross_ratio} and~\ref{subs:transverse}.
Sage~\cite{Sage} was chosen
because its open-source nature makes it particularly easy
for others to verify and reuse the code below;
there are of course many commercial packages
that can perform the same calculations.

In the listing below, lines beginning with the
`\verb+#+'
character are comments.

\bigskip
{\scriptsize

\begin{Verbatim}[commandchars=\%\&\$]
# Sage code for some calculations in R_3

# Part 1: look at the standard forms and the quotient map,
# as described in Sections %ref&subs:standard$ and %ref&subs:cross_ratio$.

# First make a number field with our `special value':
# alpha=e^{\pi i/3} is a root of z^2-z+1,
# but it's more convenient to express things in terms of a=-(1+alpha),
# which is a root of z^2+3z+3.
# Use `aa' for this specific value of a.
# In Sage, `QQ' represents the rational numbers;
# we define `RR' to be the quadratic field extension containing `aa'.
var('z')
RR.<aa>=QQ.extension(z^2+3*z+3)
alpha=-(1+aa)

# Now we need a polynomial ring:
R.<x,t,a,x1,x2,x3,x4,x5,x6>=RR[]
# Use x1-x6 as arguments of symmetric functions: see below.
# Note that the symbol R isn't used below;
# the point of the previous command is to define the variables x,t,a,...
# and the rules for factorising polynomials in those variables.

# The standard form of an R3 element, as in Lemma %ref&lemma:standard_form$:
f=x^2*(x+a)/((2*a+3)*x-a-2)

# Relationships between a, mu, lambda given by Lemma %ref&lemma:standard_form$.
# Note that `lambda' is a reserved word in the Python programming language,
# so we shouldn't use it as the name of a variable,
# hence the names `lambdaval' etc.
def muval(aval):
    return -aval*(aval+2)/(2*aval+3)

def lambdaval(aval):
    return muval(aval)^3/(aval+2)^2

def aval(muval,lambdaval):
    return (muval^3+3*muval*lambdaval-4*lambdaval)/(2*lambdaval*(1-muval))

# Fourth critical point and critical value:
m=muval(a)
l=lambdaval(a)
# `m' and `l' are abbreviations for mu and lambda respectively.

#  What happens to a when we shuffle lambda and mu?
f0=a
f1=aval(1/m,1/l)
f2=aval(1-m,1-l)
f3=aval(m/(m-1),l/(l-1))
f4=aval(1/(1-m),1/(1-l))
f5=aval((m-1)/m,(l-1)/l)
# f0 through f5 are the expressions of Lemma %ref&lemma:a_transform$.

# What about symmetric functions of those f0-f5?
# Define s1-s6 to be the elementary symmetric functions of x1-x6
# Oddly, there doesn't seem to be a Sage builtin that does it,
# so I need to create them by hand.
# It's easy to make s1 and s6: just add/multiply all the variables.
varlist=[x1,x2,x3,x4,x5,x6]
s1=sum(varlist)
s6=prod(varlist)
# For s2 through s5, we use the `combinations' function to generate
# the needed monomials.
# Unfortunately `combinations' can't take variables as arguments,
# so we need make a list of coefficients
# and then substitute the variables `by hand'.
otherfuncs=[0,0,0,0,0,0]
for i in [2,3,4,5]:
    coefflist=combinations(range(6),i)
    for index in coefflist:
        otherfuncs[i]=otherfuncs[i]+prod(varlist[n] for n in index)
s2=otherfuncs[2]
s3=otherfuncs[3]
s4=otherfuncs[4]
s5=otherfuncs[5]
# Now s1 through s6 contain the elementary symmetric functions
# of the variables x1 through x6.
# We'll use this to verify %eqref&eq:pi$.
# A similar process applied to the lambda values of %eqref&eq:sixXratios$
# gives %eqref&eq:symmXratio$.

# Calculate symmetric functions of the six a-values f1 through f6:
s1a=s1.subs(x1=f1,x2=f2,x3=f3,x4=f4,x5=f5,x6=f0)
s2a=s2.subs(x1=f1,x2=f2,x3=f3,x4=f4,x5=f5,x6=f0)
s3a=s3.subs(x1=f1,x2=f2,x3=f3,x4=f4,x5=f5,x6=f0)
s4a=s4.subs(x1=f1,x2=f2,x3=f3,x4=f4,x5=f5,x6=f0)
s5a=s5.subs(x1=f1,x2=f2,x3=f3,x4=f4,x5=f5,x6=f0)
s6a=s5.subs(x1=f1,x2=f2,x3=f3,x4=f4,x5=f5,x6=f0)

# Write these in lowest terms;
# Sage needs a little help here.
s2a=(s2a.numerator()/2^20)/(s2a.denominator()/2^20)
s3a=(s3a.numerator()/2^40)/(s3a.denominator()/2^40)
s4a=(s4a.numerator()/2^40)/(s4a.denominator()/2^40)
s5a=(s5a.numerator()/2^20)/(s5a.denominator()/2^20)
s6a=(s6a.numerator()/2^20)/(s6a.denominator()/2^20)
# Results:
# s2a=\frac{-a^{6}-9a^{5}+135a^{3}+360a^{2}+351a+117}
#          {a^{4}+6a^{3}+13a^{2}+12a+4}
# s1a=-144/16
# 6s2a+s3a=135
# 13s2a-s4a=360
# 12s2a+s5a=351
# s5a=s6a

# Something strange: s2a(0)=s2a(-3)=s2a(-3/2)=117/4
# Can s2a send anything else to 117/4?
s2apoly=117*s2a.denominator()-4*s2a.numerator()
factor(s2apoly)
# The result is 4*a^2*(a+3/2)^2*(a+3)^2
# so the answer is no, those are the only three values sent to 117/4.

# Does s2a distinguish the orbits?
calcdiff=s2a.subs(a=x1)-s2a.subs(a=x2)
factor(calcdiff)
# This produces the equation of Proof %ref&lemma:same_pi$.

# Part 2: Find vectors spanning tangent spaces to orbits of R_3,
# as described in Section %ref&subs:transverse$.

# Need to redefine the polynomial ring R,
# since Sage isn't comfortable mixing the algebraic number aa
# with the transcendental number e.
R.<x,t,a,b,K>=QQ[]

# Infinitesimal generators for the tangent space
eat=x*e^(2*t)
# `eat' means e^{at} where a is an element of the Lie algebra
# as in Proof %ref&prop:transverse$; similarly ebt and ect.
ebt=x+t
ect=x/(t*x+1)

def find_tangents(gamma):
# Given a rational function gamma, calculate the derivatives,
# evaluated at t=0, of gamma composed with each of eat, ebt, ect
# for both pre- and postcomposition.
# nb the constant term of the denominator of gamma must be nonzero.
    answerlist=[]
    for funcpair in [(gamma,eat),(gamma,ebt),(gamma,ect),
                     (eat,gamma),(ebt,gamma),(ect,gamma)]:
        comp=symbolic_expression(funcpair[0].subs(x=funcpair[1]))
           # need to convert to symbolic_expression
           # in order to use rational_simplify
        comp=comp.rational_simplify()
        numerator_coeffs=[0,0,0,0]
        denominator_coeffs=[0,0,0,0]
        for term in comp.numerator().coeffs(x):
            numerator_coeffs[term[1]]=term[0]
        numerator_coeffs.reverse()
        for term in comp.denominator().coeffs(x):
            denominator_coeffs[term[1]]=term[0]
        denominator_coeffs.reverse()
        coefflist=[term/denominator_coeffs[-1]
                    for term in numerator_coeffs+denominator_coeffs[:-1]]
        compd=[term.derivative(t).subs(t=0) for term in coefflist]
        answerlist.append(compd)
    return answerlist

eta_zero=(x^3-a*x)/(-b*x^2+1) # This is the function of %eqref&eq:eta$.
find_tangents(eta_zero) # This outputs the matrix of Proof %ref&prop:transverse$.
\end{Verbatim}

}
\bigskip

\section{Index of notation} \label{appendix:notation}

\begin{longtable}{rl}
$R_d$ & space of rational functions of degree $d$ \\
$G$ & the group $\mbsgp\times\mbsgp$, acting on $R_3$ \\
$\pi$ & the categorical quotient map $R_3\to\complex$ \\
$f$, $g$ & denote elements of $R_3$ \\
$g$ & sometimes denotes an element of $G$, \\
    & and sometimes a Möbius transformation \\
$\alpha$, $\beta$ & denote Möbius transformations \\
$(z_1,z_2;z_3,z_4)$ & cross-ratio \\
$\lambda$, $\mu$ & values of cross-ratio \\
$\sigma_k$ & elementary symmetric functions \\
$s_k$, $s$ & symmetrised versions of cross-ratio \\
$a$ & the parameter of Lemma~\ref{lemma:standard_form} \\
$f_a$ & the function of Definition~\ref{def:standard_form} \\
$f^g$ & the image of $f\in R_3$ under the action of $g\in G$ \\
$f^{(\sigma)}$ & see Lemma~\ref{lemma:signature} \\
$\sigma$, $\rho$ & orderings of critical values and points (Section~\ref{subs:cross_ratio}) \\
$\eta$, $\eta_0$ & the maps $\complex^2\to\pseven$ of \eqref{eq:eta}
   and following \\
$\exp$ & the exponential map for $G$ \\
$\phi$, $\psi$ & the two maps that are composed in Proposition~\ref{prop:composition} \\
$\omega$ & the ``Buzzard--Lu map'', defined after Corollary~\ref{cor:surjective} \\
$\chi$ & the first component of $\psi$, defined in Proposition~\ref{prop:chi_surjective} \\
$A$, $B$, $C$ & infinitesimal generators of the Lie algebra $\Liesl_2(\complex)$ \\
\end{longtable}

\end{document}